\documentclass[5p,preprint]{elsarticle}

\usepackage{setspace}
\usepackage{float}
\usepackage[T1]{fontenc}
\usepackage{mathptmx}
\usepackage{amsthm}
\usepackage[active]{srcltx}
\usepackage{amsmath}
\usepackage{amssymb}
\usepackage{amsfonts}
\usepackage{mathdots}
\usepackage{oldgerm}
\usepackage{mathrsfs}
\usepackage{comment}
\usepackage{setspace}
\usepackage{tikz}
\usepackage{pgfplots}
\usepackage[scanall]{psfrag}
\usepackage{graphicx}
\usepackage{pst-all}
\usepackage{flushend}
\usepackage{subfig}
\usepackage{mdwlist}
\usepackage{color}
\usepackage{longtable}
\usepackage{verbatim}
\usepackage{siunitx}
\usetikzlibrary{decorations.pathreplacing}

\newcommand{\cC}{\mathcal{C}}

\newcommand{\cW}{\mathcal{W}}

\newcommand{\cL}{\mathcal{L}}

\newcommand{\cD}{\mathcal{D}}

\newcommand{\cF}{\mathcal{F}}

\usetikzlibrary{arrows}
\usetikzlibrary{fit}\usetikzlibrary{calc}
  \pgfdeclarelayer{background}
  \pgfsetlayers{background,main}

\theoremstyle{theorem}
\newtheorem{Prop}{Proposition}[section]

\newtheorem{Thm}[Prop]{Theorem}

\theoremstyle{definition}

\newcommand{\RN}[1]{\uppercase\expandafter{\romannumeral#1}}
\newcommand{\eps}{\varepsilon}

\newcommand{\N}{{\mathbb{N}}}

\newcommand{\R}{{\mathbb{R}}}

\newcommand{\vp}{\varphi}

\newcommand{\setdef}[2]{\left\{\ #1\ \left|\ \vphantom{#1} #2\ \right.\right\}}

\newcommand{\ds}[1]{{\rm \, d} #1 \,}

\DeclareMathOperator{\sgn}{sgn}
\DeclareMathOperator{\loc}{loc}
\DeclareMathOperator{\erf}{erf}

{\left(\begin{smallmatrix}}
{\end{smallmatrix}\right)}

{\left[\begin{smallmatrix}}
{\end{smallmatrix}\right]}

\newlength{\innersep}
\newlength{\maxlength}
\newlength{\dummylength}

\newcommand{\JordanBlock}[3]{
\setlength{\arraycolsep}{0pt}
\renewcommand{\arraystretch}{0}
\settowidth{\maxlength}{$#1$}
\settoheight{\dummylength}{$#1$}
\ifdim\dummylength>\maxlength
  \setlength{\maxlength}{\dummylength}
\fi
\settowidth{\dummylength}{$#2$}
\ifdim\dummylength>\maxlength
  \setlength{\maxlength}{\dummylength}
\fi
\settoheight{\dummylength}{$#2$}
\ifdim\dummylength>\maxlength
  \setlength{\maxlength}{\dummylength}
\fi
\setlength{\innersep}{0.1\maxlength}
\addtolength{\maxlength}{\innersep}
\addtolength{\maxlength}{\innersep}
\newcommand{\invisiblebox}{\phantom{\rule{\maxlength}{\maxlength}}}
\begin{array}{cccc}
  \tikz[remember picture] \node[outer sep=0,inner sep=\innersep] (a11) {$#1$}; &{\tikz[remember picture] \node[outer sep=0,inner sep=\innersep] (a21) {$#2$};}& & \invisiblebox\\
  &&\phantom{\rule{#3}{#3}} &\\
   \invisiblebox &\invisiblebox& &{\tikz[remember picture] \node[outer sep=0,inner sep=\innersep] (a43) {$#2$};} \\
   \invisiblebox &\invisiblebox& &
   {\tikz[remember picture] \node[outer sep=0,inner sep=\innersep] (a44) {$#1$};}
\end{array}
\tikz[remember picture, overlay] \draw (a11) edge[very thick] (a44);
\tikz[remember picture, overlay] \draw (a21) edge[very thick] (a43);
}

\begin{document}

\begin{frontmatter}

\title{Funnel cruise control\tnoteref{thanks}}
\tnotetext[thanks]{This work was supported by the German Research Foundation (Deutsche Forschungsgemeinschaft) via the grant BE 6263/1-1.}

\author[Paderborn]{Thomas Berger}\ead{thomas.berger@math.upb.de}
\author[Hamburg]{Anna-Lena Rauert}\ead{anna-lena.rauert@web.de}

\address[Paderborn]{Institut f\"ur Mathematik, Universit\"at Paderborn, Warburger Str.~100, 33098~Paderborn, Germany}
\address[Hamburg]{Haldesdorfer Str.~117, 22179~Hamburg, Germany}


\begin{abstract}
We consider the problem of vehicle following, where a safety distance to the leader vehicle is guaranteed at all times and a favourite velocity is reached as far as possible. We introduce the funnel cruise controller as a novel universal adaptive cruise control mechanism which is model-free and achieves the aforementioned control objectives. The controller consists of a velocity funnel controller, which directly regulates the velocity when the leader vehicle is far away, and a distance funnel controller, which regulates the distance to the leader vehicle when it is close so that the safety distance is never violated. We provide a rigorous proof for the feasibility of the overall controller design. The funnel cruise controller is illustrated by a simulation of three different scenarios which may occur in daily traffic.
\end{abstract}

\end{frontmatter}


%
\section{Introduction}\label{Sec:Intr}
%

With traffic steadily increasing, simple cruise control (see e.g.~\cite{AstrMurr08}), which holds the velocity on a constant level, becomes less useful. A controller which additionally allows a
vehicle to follow the vehicle in front of it while continually adjusting speed to maintain a safe distance is a suitable alternative. Various methods which achieve this are available in the literature, see e.g.\ the survey~\cite{XiaoGao10} on adaptive cruise control systems. A common method is the use of proportional-integral-derivative (PID) controllers, see e.g.~\cite{AstrMurr08,IoanXu93,IoanChie93,YanaKane98}, which however are not able to guarantee any safety.

Another popular method is model predictive control (MPC), where the control action is defined by repeated solution of a finite-horizon optimal control problem. A two-mode MPC controller is developed in~\cite{BageGarr04}, where the controller switches between velocity and distance control. The MPC method introduced in~\cite{LiLi11} incorporates the fuel consumption and driver desired response in the cost function of the optimal control problem. In~\cite{MagdAlth17} a method which guarantees both safety and comfort is developed. It consists of a nominal controller, which is based on MPC, and an emergency controller which takes over when MPC does not provide a safe solution.

Control methods based on control barrier functions which penalize the violation of given constraints have been developed in~\cite{AmesGriz14,MehrMa15}. While safety constraints are automatically guaranteed by this approach, it may be hard to find a suitable control barrier function. Furthermore, these methods are implemented as open-loop control inputs using quadratic programs, which require knowledge of the model parameters and hence are not robust in general. Another recent method is correct-by-construction adaptive cruise control~\cite{NilsHuss16}, which is also able to guarantee safety. However, the computations are based on a so called finite abstraction of the system (which is already expensive) and changes of the system parameters require a complete re-computation of the finite abstraction.

Drawbacks of the aforementioned approaches are that either safety cannot be guaranteed (as in~\cite{IoanXu93,IoanChie93,YanaKane98}) or the model must be known exactly (as in~\cite{AmesGriz14,BageGarr04,LiLi11,MagdAlth17,MehrMa15,NilsHuss16}). However, the requirements on driver assistance systems are increasing steadily. It is expected that in the near future autonomous vehicles will completely take over all driving duties. Therefore, a cruise control mechanism is desired which achieves both: under any circumstances (in particular, in emergency situations) the prescribed safety distance to the preceding vehicle is guaranteed and at the same time the parameters of the model, such as aerodynamic drag or rolling friction, need {\it not} be known exactly, i.e., the control mechanism is model-free. The latter property also guarantees that the controller is inherently robust, in particular with respect to uncertainties, modelling errors or external disturbances. Another requirement on the controller is that it should be simple in its design and of low complexity, and that it only requires the measurement of the velocity and the distance to the leading vehicle. We stress that in a lot of other approaches as e.g.~\cite{BageGarr04,IoanChie93,LiLi11,MagdAlth17} the position, velocity and/or acceleration of the leading vehicle must be known at each time.

In the present work we propose a novel control design which satisfies the above requirements. Our control design is based on the funnel controller  which was developed in~\cite{IlchRyan02b}, see also the survey~\cite{IlchRyan08} and the recent paper~\cite{BergHoan18}. The funnel controller is a low-complexity model-free output-error feedback of high-gain type. The funnel controller is an adaptive controller since the gain is adapted to the actual needed value by a time-varying (non-dynamic) adaptation scheme. It has been successfully applied e.g.\ in temperature control of chemical reactor models~\cite{IlchTren04}, control of industrial servo-systems~\cite{Hack17} and underactuated multibody systems~\cite{BergOtto19}, speed control of wind turbine systems~\cite{Hack14,Hack15b,Hack17}, current control for synchronous machines~\cite{Hack15a,Hack17}, DC-link power flow control~\cite{SenfPaug14}, voltage and current control of electrical circuits~\cite{BergReis14a}, oxygenation control during artificial ventilation therapy~\cite{PompAlfo14} and control of peak inspiratory pressure~\cite{PompWeye15}.

In our design we will distinguish two different cases. If the preceding vehicle is far away, i.e., the distance to it is larger than the safety distance plus some constant, then a velocity funnel controller will be active which simply regulates the velocity of the vehicle to the desired pre-defined velocity. If the preceding vehicle is close, then a distance funnel controller will be active which regulates the distance to the preceding vehicle to stay within a predefined performance funnel in front of the safety distance. The combination of these two controllers results in a funnel cruise controller which guarantees safety at all times. We like to stress that the distance funnel controller does not directly regulate the position of the vehicle, but a certain weighting between position and velocity; hence, a relative degree one controller suffices.

A conference proceedings version of the present paper has been published in~\cite{BergRaue18}. In the present journal version we consider the effect of additional disturbances in the model and have added a full proof of the main result in Theorem~\ref{Thm:funnel-acc}. Furthermore, we explain the controller design in more detail and discuss the presence of control constraints.

\subsection{Nomenclature}\label{Ssec:Nomencl}

\noindent{
\begin{tabular}[ht]{p{60pt}p{165pt}}
$\R_{\ge 0}$ &$=[0,\infty)$\\
$\mathcal{L}^\infty_{\loc}(I\!\to\!\R^n)$
  &  the set of locally essentially bounded functions $f:I\!\to\!\R^n$, $I\subseteq\R$ an interval\\
$\mathcal{L}^\infty(I\!\to\!\R^n)$
  &  the set of essentially bounded functions $f:I\!\to\!\R^n$\\
$\|f\|_\infty$ & = ${\rm ess\ sup}_{t\in I} \|f(t)\|$\\
$\mathcal{W}^{k,\infty}(I\!\to\!\R^n)$
  &  the set of $k$-times weakly differentiable functions $f:I\!\to\!\R^n$ such that $f,\ldots, f^{(k)}\in \mathcal{L}^\infty(I\!\to\!\R^n)$\\
$\mathcal{C}(V\!\to\!\R^n) $
  &  the set of continuous functions $f:V\!\to\! \R^n$, $V\subseteq\R^m$\\
$\left.f\right|_{W}$ & restriction of the function $f:V\!\to\!\R^n$ to $W\subseteq V$
\end{tabular}
}

\subsection{Framework and system class}\label{Ssec:SysClass}

In the present work we consider the framework of one vehicle following another, see Fig.~\ref{Fig:Framework}.

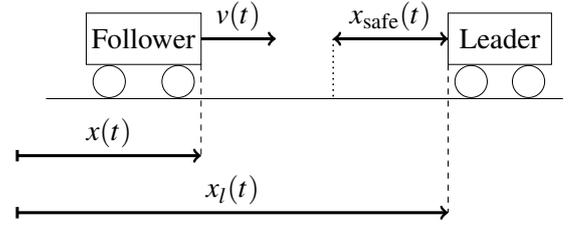
\begin{figure}[h!t]
  \centering
\resizebox{0.4\textwidth}{!}{
\begin{tikzpicture}

\draw (7,5) -- (16,5); \draw[dotted,line width=0.8pt] (12,6.05) -- (12,5);

\draw (14.4,5.3) circle (8pt); \draw (15.4,5.3) circle (8pt);\draw (14,5.6) rectangle (15.8,6.5); 
\draw (8.1,5.3) circle (8pt); \draw (9.3,5.3) circle (8pt); \draw (7.7,5.6) rectangle (9.7,6.5); 

\draw [<->,line width=1.5pt] (14,6.05)--node[above] {\Large $x_{\rm safe}(t)$}(12,6.05);

\draw[->,line width=1.5pt] (9.7,6.05) -- (11,6.05)  node[midway, above]{\Large $v(t)$};

\draw[dashed] (9.7,5.6) -- (9.7,4); \draw[dashed] (14,5.6) -- (14,3);
\draw[->,line width=1.5pt] (6.5,4) -- (9.7,4)  node[midway, above]{\Large $x(t)$};
\draw[->,line width=1.5pt] (6.5,3) -- (14,3)  node[midway, above]{\Large $x_l(t)$};
\draw[line width=1.5pt] (6.5,4.1) -- (6.5,3.9);
\draw[line width=1.5pt] (6.5,3.1) -- (6.5,2.9);

\node at (14.9,6.05) {\Large Leader};
\node at (8.7,6.05) {\Large Follower};
\end{tikzpicture}
}
\caption{Vehicle following framework.}
\label{Fig:Framework}
\end{figure}

By~$x_l$ we denote the position of the leader vehicle, while~$x$ and~$v$ denote the position and velocity of the follower vehicle. The change in momentum of the latter is given by the difference of the force~$F$ generated by the contact of the wheels with the road and the forces due to gravity~$F_g$ (including the changing slope of the road), the aerodynamic drag~$F_a$ and the rolling friction~$F_r$. Detailed modelling of these forces can be very complicated since all the individual components of the vehicle have to be taken into account. Therefore, we use the following simple models which are taken from~\cite[Sec.~3.1]{AstrMurr08}:
\begin{align*}
  F_g:&\ \R_{\ge 0}\to\R,\ t\mapsto m g \sin \theta(t),\\
  F_a:&\ \R_{\ge 0}\times \R\to\R,\ (t,v)\mapsto \tfrac12 \rho(t) C_d A v^2,\\
  F_r:&\ \R\to\R,\ v\mapsto m g C_r \sgn(v),
\end{align*}
where $m$ (in \si{\kilo\gram}) denotes the mass of the (following) vehicle, $g = \SI{9.81}{\metre\per\second^2}$ is the acceleration of gravity,~$\theta(t)\in [-\tfrac{\pi}{2}\si{\radian}, \tfrac{\pi}{2}\si{\radian}]$ and $\rho(t)$ (in \si{\kilo\gram\per\metre^3}) denote the slope of the road and the (bounded) density of air at time~$t$, resp.,~$C_d$ denotes the (dimensionless) shape-dependent aerodynamic drag coefficient and~$C_r$ the (dimensionless) coefficient of rolling friction, and~$A$ (in \si{\metre^2}) is the frontal area of the vehicle.

Since the discontinuous nature of the rolling friction causes some problems in the theoretical treatment and in the vehicle following framework the velocities are typically positive, we approximate the~$\sgn$ function by the smooth error function
\[
    \erf(z) = \frac{2}{\sqrt{\pi}} \int_0^z e^{-t^2}\ds{t},\quad z\in\R,
\]
using the property that
\[
    \forall\, z\in\R:\ \lim_{\alpha\to \infty} \erf(\alpha z) = \sgn(z).
\]
Therefore, we will use the following model for the rolling friction:
\begin{equation}\label{eq:slidingfriction2}
     F_r:\R\to\R,\ v\mapsto m g C_r \erf(\alpha v)
\end{equation}
for sufficiently large parameter~$\alpha>0$. For more sophisticated friction models we refer to~\cite{ArmsDupo94,LeinNijm04}.

The force~$F$ which is generated by the engine of the vehicle is usually given as torque curve (depending on the engine speed) times a signal which controls the throttle position, see~\cite[Sec.~3.1]{AstrMurr08}. Since the latter can be calculated from any given force~$F$ and velocity~$v$ (taking the current gear into account), here we assume that we can directly control the force~$F$, i.e., the control signal is $u(t) = F(t)$. The equations of motion for the vehicle are then given by
\begin{equation}\label{eq:Sys}
\begin{aligned}
    \dot x(t) &= v(t),\\
    m \dot v(t) &= u(t) - F_g(t) - F_a\big(t,v(t)\big) - F_r\big(v(t)\big) +\delta(t),
\end{aligned}
\end{equation}
with the initial conditions
\begin{equation}\label{eq:IC}
  x(0) = x^0\in\R,\quad v(0) = v^0\in\R,
\end{equation}
where $\delta\in \cL^\infty(\R_{\ge 0}\to\R)$ is a bounded disturbance which captures modelling errors, uncertainties and noises, which may be caused by unexpected potholes in the road for instance.

\subsection{Control objective}\label{Ssec:ContrObj}

Roughly speaking, the control objective is to design a control input~$u(t)$ such that~$v(t)$ is as close to a given favourite speed~$v_{\rm ref}(t)$ as possible, while at the same time a safety distance to the leading vehicle is guaranteed, i.e., $x_l(t) - x(t) \ge x_{\rm safe}(t)$. The safety distance $x_{\rm safe}(t)$ should prevent collision with the leading vehicle and is typically a function of the vehicle velocity, but could also be a constant or a function of other variables. In the literature different concepts are used, see e.g.~\cite{HongPark16,SantRaja03} and the references therein. A common model for the safety distance that we also use in the present paper is
\begin{equation}\label{eq:xsafe}
    x_{\rm safe}(t) = \lambda_1 v(t) + \lambda_2
\end{equation}
with positive constants~$\lambda_1$ (in \si{\second}) and~$\lambda_2$ (in \si{\metre}). The parameter~$\lambda_1$ models the time gap between the leader and follower vehicle and~$\lambda_2$ is the minimal distance when the velocity is zero. If for instance $\lambda_1 = \SI{0.5}{\second}$, then it would take the following vehicle $\SI{0.5}{\second}$ to arrive at the leading vehicle's present position.

We assume that the distance~$x_l(t) - x(t)$ to the leader vehicle as well as the velocity~$v(t)$ can be measured, i.e., they are available for the controller design. Apart from that, the controller design should be model-free, i.e., knowledge of the parameters~$m, \theta(t), \rho(t), C_d, C_r,$ and~$A$ as well as of the initial values~$x^0, v^0$ and the disturbance~$\delta(t)$ is not required. This makes the controller robust to modelling errors, uncertainties, noise and disturbances. Summarizing, the objective is to design a (nonlinear and time-varying) control law of the form
\begin{equation}\label{eq:u(t)}
    u(t) = F\big( t, v(t), x_l(t) - x(t)\big)
\end{equation}
such that, when applied to a system~\eqref{eq:Sys}, in the closed-loop system we have that for all $t\ge 0$
\begin{enumerate}
  \item[(O1)] $x_l(t) - x(t) \ge x_{\rm safe}(t)$,
  \item[(O2)] $|v(t) - v_{\rm ref}(t)|$ is as small as possible such that~(O1) is not violated.
\end{enumerate}

\subsection{Funnel control for relative degree one systems}\label{Ssec:FunConRD1}

The final control design will consist of two different funnel controllers for appropriate relative degree one systems. While, in view of the control objective, the system~\eqref{eq:Sys} cannot be rewritten as a relative degree one system, this is possible when velocity and distance control are considered separately. This separate consideration may serve as a motivation and therefore we briefly recall the concept of funnel control here.

The first version of the funnel controller was developed in~\cite{IlchRyan02b} and this version will be sufficient for our purposes. We consider nonlinear relative degree one systems governed by functional differential equations of the form
\begin{equation}\label{eq:FuncSys}
    \begin{aligned}
    \dot y(t) &= f\big(d(t), (Ty)(t)\big) + \gamma\, u(t),\\
    y(0) &= y^0\in\R,
    \end{aligned}
\end{equation}
where $\gamma>0$ is the high-frequency gain and
\begin{itemize}
  \item $d\in\cL^\infty(\R_{\ge 0}\to\R^p)$, $p\in\N$, is a disturbance;
  \item $f\in\cC(\R^p\times\R^q\to\R)$, $q\in\N$;
  \item $T:\cC(\R_{\ge 0}\to\R)\to\cL^\infty_{\loc}(\R_{\ge 0}\to\R^q)$ is an operator with the following properties:
      \begin{enumerate}
      \item[a)] $T$ maps bounded trajectories to bounded trajectories, i.e, for all $c_1>0$, there exists $c_2>0$ such that for all $\zeta\in \mathcal{C}(\R_{\ge 0}\rightarrow\R)\cap\cL^\infty(\R_{\ge 0}\rightarrow\R)$ we have $T(\zeta)\in \cL^\infty(\R_{\ge 0}\rightarrow\R^q)$ and
        \[ \|\zeta\|_\infty \le c_1\ \ \Longrightarrow\ \ \|T(\zeta)\|_\infty \le c_2.\]
      \item[b)] $T$ is causal, i.e., for all $t\geq 0$ and all $\zeta,\xi\in \cC(\R_{\ge 0}\to\R)$:
      \[
            \left.\zeta\right|_{[0,t]} = \left.\xi\right|_{[0,t]}\ \ \Longrightarrow\ \ \left.T(\zeta)\right|_{[0,t]} \stackrel{\rm a.e.}{=} \left.T(\xi)\right|_{[0,t]},
      \]
      where ``a.e.'' stands for ``almost everywhere''.
      \item[c)] $T$ is locally Lipschitz continuous in the following sense: for all $t\geq 0$ and all $\xi \in \cC([0,t]\to\R)$ there exist $\tau,\delta, c>0$ such that, for all $\zeta_1,\zeta_2 \in \cC(\R_{\ge 0}\to\R)$ with $\left.\zeta_i\right|_{[0,t]}=\xi$ and $|\zeta_i(s) - \xi(s)| < \delta$  for all $s\in[t,t+\tau]$ and $i=1,2$, we have
          \[
            \left\| \left.\big(T(\zeta_1)\!-\!T(\zeta_2)\big)\right|_{[t,t+\tau]}\right\|_\infty \!\!\le\! c \left\| \left.\big(\zeta_1\!-\! \zeta_2\big)\right|_{[t,t+\tau]}\right\|_\infty\!.
          \]
      \end{enumerate}
\end{itemize}
The functions $u,y:\R_{\geq 0}\to \R$ are called input and output of the system \eqref{eq:FuncSys}, resp. In~\cite{BergHoan18, IlchRyan09, IlchRyan02b, IlchRyan07} it is shown that the class of systems~\eqref{eq:FuncSys} encompasses linear and nonlinear systems with strict relative degree one and bounded-input, bounded-output stable internal dynamics. The operator~$T$ allows for infinite-dimensional (linear) systems, systems with hysteretic effects or nonlinear delay elements, and combinations thereof; $T$ is typically the solution operator corresponding to a (partial) differential equation which describes the internal dynamics of the system. In~\cite{IlchRyan02b, IlchRyan07} linear infinite-dimensional systems in a special Byrnes-Isidori form are considered, which is discussed in detail in~\cite{IlchSeli16}. More general classes involving nonlinear equations and unbounded operators are discussed in~\cite{BergPuch19,BergPuch19b}.

The funnel controller for systems~\eqref{eq:FuncSys} is of the form
\begin{equation}\label{eq:fun-con}
\boxed{
\begin{aligned}
    u(t) &= - k(t) e(t),\ &e(t) &= y(t) - y_{\rm ref}(t),\\
    k(t) &= \frac{1}{1-\varphi(t)^2 e(t)^2},
\end{aligned}
}
\end{equation}
where $y_{\rm ref}\in \cW^{1,\infty}(\R_{\ge 0}\to\R)$ is the reference signal, and guarantees that the tracking error~$e(t)$ evolves within a prescribed performance funnel
\begin{equation}
\mathcal{F}_{\varphi} := \setdef{(t,e)\in\R_{\ge 0} \times\R}{\varphi(t) |e| < 1},\label{eq:perf_funnel}
\end{equation}
which is determined by a function~$\varphi$ belonging to
\[
\Phi \!:=\!
\left\{
\!\varphi\!\in\!\cW^{1,\infty}(\R_{\ge 0}\!\to\!\R)
\left|\!\!\!
\begin{array}{l}
\text{$\varphi(s)>0$ for all $s>0$ and}\\\text{for all }\eps>0:\\
 \left.\!(\!1\!/\!\varphi\!)\right|_{[\eps,\infty)}\!\!\in\!\cW^{1,\infty}([\eps,\infty)\!\to\!\R)
\end{array}
\right.\!\!\!\!\!
\right\}\!\!.
\]
Thus, if the norm of the error~$e$ draws close to the funnel boundary~$1/\varphi$, the gain~$k$ increases and hence, the control is more aggressive. Otherwise, if the error~$e$ is close to zero, then the control is more relaxed with small gains~$k$.

The funnel boundary is given by the reciprocal of $\varphi$, see Fig.~\ref{Fig:funnel}. The case $\varphi(0)=0$ is explicitly allowed, meaning that no restriction is put on the initial value since $\varphi(0) |e(0)| < 1$; the funnel boundary $1/\varphi$ has a pole at $t=0$ in this case.

\begin{figure}[h]
\vspace*{-3mm}
\hspace{6mm}\includegraphics[width=8cm, trim=170 550 230 120, clip]{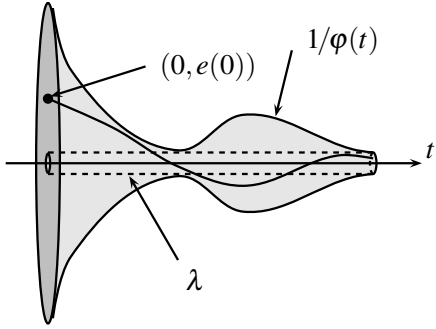}
\vspace*{-2mm}
\caption{Error evolution in a funnel $\mathcal F_{\varphi}$ with boundary $1\!/\!\varphi(t)$.}
\label{Fig:funnel}
\end{figure}

An important property is that each performance funnel $\mathcal{F}_{\varphi}$ is bounded away from zero, i.e., because of boundedness of $\varphi$ there exists $\lambda>0$ such that $1/\varphi(t)\geq\lambda$ for all $t > 0$. Thus, asymptotic accuracy, i.e., $\lim_{t\to\infty} e(t)=0$, is not possible in general, but the tracking error~$e(t)$ can be made as small as desired. We stress that the funnel boundary is not necessarily monotonically decreasing. Widening the funnel over some later time interval might be beneficial, e.g., when periodic disturbances are present or the reference signal changes strongly. For typical choices of funnel boundaries see e.g.~\cite[Sec.~3.2]{Ilch13}.

In~\cite{IlchRyan02b}, the existence of global solutions of the closed-loop system~\eqref{eq:FuncSys},~\eqref{eq:fun-con} is investigated. To this end, $y:[0,\omega)\to\R$, $\omega\in(0,\infty]$, is called a \emph{solution} of~\eqref{eq:FuncSys},~\eqref{eq:fun-con}, if $y(0) = y^0$ and $y$ is weakly differentiable and satisfies the differential equation in~\eqref{eq:FuncSys} with~$u$ as in~\eqref{eq:fun-con} for almost all $t\in[0,\omega)$; $y$ is called \emph{maximal}, if it has no right extension that is also a solution. Note that uniqueness of solutions of~\eqref{eq:FuncSys},~\eqref{eq:fun-con} is not guaranteed in general.

The following result is proved in~\cite{IlchRyan02b}.

\begin{Thm}\label{Thm:funnel}
Consider a system~\eqref{eq:FuncSys} with initial value $y^0\in \R$, a reference signal $y_{\rm ref}\in \cW^{1,\infty}(\R_{\ge 0}\to\R)$ and a funnel function $\varphi\in \Phi$ such that
\[
    \varphi(0) | y^0(0) - y_{\rm ref}(0)| < 1.
\]
Then the controller~\eqref{eq:fun-con} applied to~\eqref{eq:FuncSys} yields a closed-loop system which has a solution, and every maximal solution $y:[0,\omega)\to\R$ has the properties:
\begin{enumerate}
  \item $\omega=\infty$;
  \item all involved signals $y, k$ and $u$ are bounded;
  \item the tracking error evolves uniformly within the performance funnel in the sense
  \[
    \exists\, \eps>0\ \forall\, t>0:\ |e(t)| \le \varphi(t)^{-1}-\eps.
  \]
\end{enumerate}
\end{Thm}

\subsection{Organization of the present paper}\label{Ssec:Org}

In Section~\ref{Sec:FCC} we present a novel funnel cruise controller which satisfies the control objectives as stated in Section~\ref{Ssec:ContrObj}. The controller is basically the conjunction of a velocity funnel controller and a distance funnel controller, both formulated for appropriate relative degree one systems. Those controllers are presented separately before the final controller design is stated and feasibility is proved. In Section~\ref{Sec:Sim} the performance of the controller is illustrated for some typical model parameters and scenarios from daily traffic. Some conclusions are given in Section~\ref{Sec:Concl}.

\section{Funnel cruise control}\label{Sec:FCC}

In this section we present our novel funnel cruise control design to achieve~(O1) and~(O2), which consists of a velocity funnel controller and a distance funnel controller. We first present those controllers separately before we state the unified controller design.

\subsection{Velocity funnel controller}\label{Ssec:VFC}

When the leader vehicle is far away we do not need to care about the control objective (O1) and simply need to regulate the velocity~$v$ to the favourite velocity~$v_{\rm ref}\in \cW^{1,\infty}(\R_{\ge 0}\to\R)$ which satisfies $v_{\rm ref}(t)\ge 0$ for all $t\ge 0$. For this purpose we can treat the velocity~$v$ as the output of system~\eqref{eq:Sys} and hence the velocity tracking error is given by $e_v(t) = v(t) - v_{\rm ref}(t)$. Since the first equation in~\eqref{eq:Sys} can be ignored in this case (it does not play a role for the input-output behavior), we may define $d(t):=\big( F_g(t) -\delta(t), \rho(t)\big)$ for $t\ge 0$ and
\[
    f_v:\R^3\to\R,\ (d_1,d_2,v)\mapsto \tfrac{1}{m} \left( d_1 + \tfrac12 d_2 C_d A v^2 + F_r(v)\right).
\]
Since~$\rho$ and~$\delta$ are assumed to be bounded we obtain that~$d$ is bounded and the second equation in~\eqref{eq:Sys} can be written as
\begin{equation}\label{eq:vdot}
    \dot v(t) =  \tfrac{1}{m} u(t) - f_v\big(d(t),v(t)\big),
\end{equation}
and hence belongs to the class of systems~\eqref{eq:FuncSys} with the identity operator~$Tv = v$. Then Theorem~\ref{Thm:funnel} yields feasibility of the velocity funnel controller
\begin{equation}\label{eq:fun-con-vel}
\begin{aligned}
    u_v(t) &= - k_v(t) e_v(t),\ &e_v(t) &= v(t) - v_{\rm ref}(t),\\
    k_v(t) &= \tfrac{1}{1-\varphi_v(t)^2 e_v(t)^2},
\end{aligned}
\end{equation}
where $\varphi_v\in\Phi$, when applied to system~\eqref{eq:Sys} with initial conditions~\eqref{eq:IC} such that $\varphi_v(0) |v^0-v_{\rm ref}(0)| < 1$.

We stress that since~$x$ has been ignored for the controller design above, the first equation in~\eqref{eq:Sys} may cause it to grow unboundedly. However, since~$x_l$ is assumed to be bounded, $x_l-x$ will eventually get small enough so that the distance funnel controller discussed in the following section becomes active. In the end, this will guarantee boundedness of~$x$ from above. Boundedness of $x$  from below is a consequence of~$v_{\rm ref}$ being non-negative.

\subsection{Distance funnel controller}\label{Ssec:DFC}

If the leader vehicle is close, then the main objective of the controller is to ensure that~(O1) is guaranteed, so that~$v(t)$ may be much smaller than~$v_{\rm ref}(t)$ if necessary. To this end, we introduce a performance funnel, defined by $\vp_d\in\Phi$ with $\vp_d(0)\neq 0$, which lies directly in front of the safety distance to the leader vehicle, see Fig.~\ref{Fig:DFC}. We set $\psi_d(\cdot) := \varphi_d(\cdot)^{-1}$.

\begin{figure}[h!t]
  \centering
\resizebox{0.48\textwidth}{!}{
\begin{tikzpicture}
\node at (5,7.7) {};
\node[draw] at (5,7) {\Large Velocity Control};
\node[draw] at (10.5,7) {\Large Distance Control};
\draw[dashed] (3,5) -- (7,5); \draw[dashed] (7,7) -- (7,0); \draw (7,5) -- (14,5);\draw (14,5) -- (16,5);
\draw[dashed] (14,7) -- (14,0); \draw[dotted,line width=0.8pt] (11.05,6) -- (11.05,0);

\draw (14.4,5.3) circle (8pt); \draw (15.4,5.3) circle (8pt);\draw (14,5.6) rectangle (15.8,6.3); 
\draw (8.1,5.3) circle (8pt); \draw (9.3,5.3) circle (8pt); \draw (7.7,5.6) rectangle (9.7,6.3); 

\draw(3,2) -- (7,2);\draw[dashed] (7,2) -- (14,2);\draw (14,2) -- (16,2);
\draw (14.4,2.3) circle (8pt); \draw (15.4,2.3) circle (8pt);\draw (14,2.6) rectangle (15.8,3.3); 
\draw (4.4,2.3) circle (8pt); \draw (5.6,2.3) circle (8pt); \draw (4,2.6) rectangle (6,3.3); 

\draw [<->,line width=1.8pt] (14,0)--node[above] {\LARGE $x_{\rm safe}(t)$}(11.05,0);
\draw[fill=black] (9,0) circle (2pt); \draw[dashed] (9,0)--node [above] {\LARGE $\psi_d(t)$} (7,0);
\node at (7.05,0) {\Large (};
\node at (11,0) {\Large )} ;

\node at (14.9,5.95) {\Large Leader};
\node at (14.9,2.95) {\Large Leader};
\node at (8.7,5.95) {\Large Follower};
\node at (5,2.95) {\Large Follower};
\end{tikzpicture}
}
\caption{Illustration of the distance funnel controller.}
\label{Fig:DFC}
\end{figure}
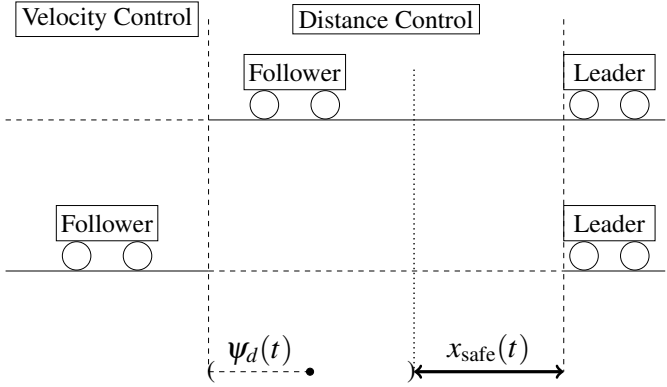

The aim is then to regulate the position~$x(t)$ to the middle of this performance funnel given by $x_l(t) - x_{\rm safe}(t) - \psi_d(t)$, where $x_l\in \cW^{1,\infty}(\R_{\ge 0}\to\R)$.

The corresponding distance tracking error is hence given by
\[
    e_d(t) = x(t) - x_l(t) + x_{\rm safe}(t) + \psi_d(t).
\]
In order to reformulate system~\eqref{eq:Sys} in the form~\eqref{eq:FuncSys} with appropriate output~$y$ and reference signal~$y_{\rm ref}$ we recall that $x_{\rm safe}(t) = \lambda_1 v(t) + \lambda_2$ and define
\begin{align*}
  y(t) &:= \lambda_1 v(t) + x(t),\\
  y_{\rm ref}(t) &:= x_l(t) - \lambda_2 - \psi_d(t).
\end{align*}
Then $e_d(t) = y(t) - y_{\rm ref}(t)$ and we further find that, invoking the first equation in~\eqref{eq:Sys},
\[
    \dot x(t) = - \tfrac{1}{\lambda_1} x(t) + \tfrac{1}{\lambda_1} y(t),\quad x(0) = x^0,
\]
hence
\[
    x(t) = e^{- \tfrac{1}{\lambda_1}t} x^0 + \int_0^t \tfrac{1}{\lambda_1} e^{- \tfrac{1}{\lambda_1}(t-s)} y(s) \ds{s} =: (T_1y)(t),\quad t\ge 0.
\]
Now define
\[
    Ty := \big( T_1y, y\big)^\top
\]
for all $y\in\cC(\R_{\ge 0}\to\R)$. It is straightforward to check that the operator $T:\cC(\R_{\ge 0}\to\R)\to\cL^\infty_{\loc}(\R_{\ge 0}\to\R^2)$, which is parameterized by~$x^0\in\R$, has the properties a)--c) as stated in Section~\ref{Ssec:FunConRD1}. Using equation~\eqref{eq:vdot} as well as~$d$ and~$f_v$ defined in Section~\ref{Ssec:VFC} we obtain
\begin{align}
  \dot y(t) &= \lambda_1 \dot v(t) + \dot x(t)\notag\\
  &= \tfrac{\lambda_1}{m} u(t) - \lambda_1 f_v\big(d(t),v(t)\big) - \tfrac{1}{\lambda_1} x(t) + \tfrac{1}{\lambda_1} y(t),\notag\\
  &= \tfrac{\lambda_1}{m} u(t) - \lambda_1 f_v\left(d(t),- \tfrac{1}{\lambda_1} (T_1y)(t) + \tfrac{1}{\lambda_1} y(t)\right)\notag\\
  &\quad - \tfrac{1}{\lambda_1} (T_1y)(t) + \tfrac{1}{\lambda_1} y(t),\notag\\
  &= \tfrac{\lambda_1}{m} u(t) - f_d\big(d(t), (Ty)(t)\big),\label{eq:ydot}
\end{align}
where
\begin{align*}
    & f_d:\R^4\to\R,\\
    & (d_1,d_2,\zeta_1,\zeta_2)\mapsto \lambda_1 f_v\left(d_1,d_2,- \tfrac{1}{\lambda_1} \zeta_1 + \tfrac{1}{\lambda_1} \zeta_2 \right) + \tfrac{1}{\lambda_1} \zeta_1 - \tfrac{1}{\lambda_1} \zeta_2.
\end{align*}
Clearly,~\eqref{eq:ydot} belongs to the class of systems~\eqref{eq:FuncSys}. Then Theorem~\ref{Thm:funnel} yields feasibility of the distance funnel controller
\begin{equation}\label{eq:fun-con-dist}
\begin{aligned}
    u_d(t) &= - k_d(t) e_d(t),\ &e_d(t) &= x(t) - x_l(t) \\
    k_d(t) &= \frac{1}{1-\varphi_d(t)^2 e_d(t)^2}, &&\quad  + x_{\rm safe}(t) + \psi_d(t),
\end{aligned}
\end{equation}
when applied to system~\eqref{eq:Sys} with initial conditions~\eqref{eq:IC} such that $\varphi_d(0) |\lambda_1 v^0 + x^0 - x_l(0) + \lambda_2 + \psi_d(0)| < 1$. Note that the assumption $\varphi_d(0)\neq 0$ is required in order to guarantee that $u_d(0)$ is well defined.

\subsection{Final control design and its feasibility}\label{Ssec:FinFunCon}

In Sections~\ref{Ssec:VFC} and~\ref{Ssec:DFC} we have seen that the separate velocity and distance funnel controllers are feasible when the initial conditions lie within the funnel boundaries at $t=0$; the control objective~(O1) is ignored in the case of velocity control and the control objective~(O2) is ignored in the case of distance control. However, it is our aim to simultaneously satisfy~(O1) and~(O2), i.e., always guarantee the safety distance and regulate the velocity to the favourite velocity as far as possible. This means that two additional scenarios must be possible:
\begin{itemize}
  \item if the follower vehicle, while using the velocity funnel controller~\eqref{eq:fun-con-vel}, enters the performance funnel in front of the safety distance, i.e., $x(t) = x_l(t) - x_{\rm safe}(t) - \psi_d(t)$, then the controller should switch to the distance funnel controller~\eqref{eq:fun-con-dist};
  \item when the distance funnel controller is active it should still be guaranteed that $v(t) < v_{\rm ref}(t) + \vp_v(t)^{-1}$, but it is possible that $v(t) \le v_{\rm ref}(t) - \vp_v(t)^{-1}$ when the leader decelerates; in the latter case it is not possible to switch back to~\eqref{eq:fun-con-vel}.
\end{itemize}

A controller which combines~\eqref{eq:fun-con-vel} and~\eqref{eq:fun-con-dist} and takes the above conditions into account faces an immediate problem: The controller~\eqref{eq:fun-con-dist} has a singularity when  $x(t) = x_l(t) - x_{\rm safe}(t) - \psi_d(t)$ since $k_d(t)\nearrow\infty$ at such points. Likewise, $k_v(t)\nearrow \infty$ for points where $v(t) = v_{\rm ref}(t) - \vp_v(t)^{-1}$, i.e., when a strong deceleration is necessary. To resolve these problems, the minimum of the control signals~$u_v(t)$ and~$u_d(t)$ is chosen in the region where the velocity performance funnel and the distance performance funnel intersect, i.e., when $(t,e_v(t))\in\cF_{\vp_v}$ and $(t,e_d(t))\in\cF_{\vp_d}$; see also Fig.~\ref{Fig:SWL} for an illustration. In particular, this guarantees that the overall funnel cruise controller given in~\eqref{eq:fun-con-fin} on the next page is continuous, since, roughly speaking, the minimum ``smoothes'' the input in the region where~$u_v$ and~$u_d$ overlap.

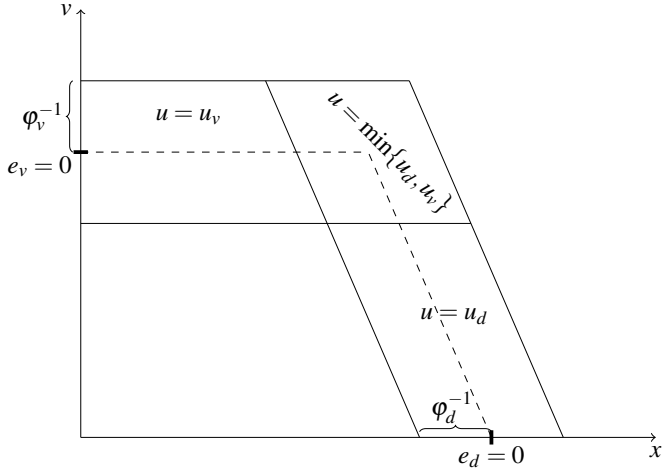
\begin{figure}[H]
  \centering
  \resizebox{0.49\textwidth}{!}{
\begin{tikzpicture}
\draw[->] (0,3) -- (0,9);
\draw (0,8) -- (4.5714,8);
\draw[dashed](0,7)--(4,7);
\node at (1.5,7.5) {$u=u_v$};
\node[left] at (0,6.8) {$e_v=0$};
\draw[line width=1.5] (0.1,7)--(-0.1,7);
\node[left] at (0,9) {$v$};
\node[rotate=-45] at (4.3,7) {$u=\min\{u_d,u_v\}$};
\draw[decorate, decoration={brace}, xshift=-0.7ex] (0,7)-- node [left] {$\varphi_v^{-1}$}(0,8);
\draw (0,6) -- (5.4286,6);
\draw(6.7143,3)--(4.5714,8);
\draw[dashed](5.7143,3)--(4,7);
\draw[line width =1.5] (5.7143,2.9)--(5.7143,3.1);
\node[below] at (5.7143,3) {$e_d=0$};
\draw(4.7143,3)--(2.5714,8);
\node at (5.2,4.7) {$u=u_d$};
\draw[decorate, decoration={brace}, yshift=0.5ex] (4.7143,3)-- node[above,pos=0.45] {$\varphi_d^{-1}$} (5.7143,3);
\draw[->] (0,3) -- (8,3);
\node[below] at (8,3) {$x$};
\end{tikzpicture}
}
\caption{Illustration of the final control design.}
\label{Fig:SWL}
\end{figure}

We emphasize again that the controller design~\eqref{eq:fun-con-fin} allows that the velocity of the follower falls out of the velocity funnel, i.e., $v(t) \le v_{\rm ref}(t) - \vp_v(t)^{-1}$ is possible when the distance funnel controller is active and the leader strongly decelerates. In order for the follower to be able to accelerate again the distance error~$e_d$ is held within~$\cF_{\varphi_d}$, but~$e_d(t)<0$ so that $u_d(t)>0$ and hence the follower accelerates -- until the velocity is back inside the velocity funnel and the velocity controller eventually takes over.

\begin{figure*}[h!tb]
\vspace*{2mm}
\begin{equation}\label{eq:fun-con-fin}
\boxed{
\begin{aligned}
u(t) &= \left\{\begin{array}{rl} -k_v(t) e_v(t), & e_d(t)\le - \vp_d(t)^{-1} \ \wedge\ (t,e_v(t))\in\cF_{\vp_v},\\[2mm]
-k_d(t) e_d(t), &  e_v(t)\le - \vp_v(t)^{-1} \ \wedge\ (t,e_d(t))\in\cF_{\vp_d},\\[2mm]
\min\{-k_v(t) e_v(t), -k_d(t) e_d(t)\}, & (t,e_v(t))\in\cF_{\vp_v}\ \wedge\ (t,e_d(t))\in\cF_{\vp_d},
\end{array}\right.\\
 k_v(t) &= \frac{1}{1-\varphi_v(t)^2 e_v(t)^2},\quad e_v(t) = v(t) - v_{\rm ref}(t),\\
 k_d(t) &= \frac{1}{1-\varphi_d(t)^2 e_d(t)^2},\quad e_d(t) = x(t) - x_l(t) + x_{\rm safe}(t) + \vp_d(t)^{-1}.
\end{aligned}
}
\end{equation}
\end{figure*}

For later purposes we define the relatively open set
\begin{equation}\label{eq:cD}
    \cD:= \cD_v\cup \cD_d\cup\cD_v^d,
\end{equation}
where
\begin{align*}
\cD_v&\!:=\!\setdef{\!\!\!\!(t,x,v)\!\in\!\R_{\ge 0}\!\times\!\R^2\!\!\!}{\small\!\!\!\!\!\!\!\begin{array}{l} \big(t,v\!-\!v_{\rm ref}(t)\big)\!\in\!\cF_{\varphi_v}\ \wedge\ \\[1mm] \big( \lambda_1 v\!+\!x\!\le\!x_l(t)\!-\!\lambda_2\!-\!2 \psi_d(t) \!\big)\end{array}\!\!\!\!\!\!\!\!}\!\!,\\
\cD_d&\!:=\!\setdef{\!\!\!\!(t,x,v)\!\in\!\R_{\ge 0}\!\times\!\R^2\!\!\!}{\small\!\!\!\!\!\!\!\begin{array}{l}\big( v\!\le\!v_{\rm ref}(t)\!-\! \varphi_v(t)^{\!-1\!}\big) \ \wedge\ \\[1mm] \big(t, \lambda_1 v\!+\!x\!-\!x_l(t)\!+\!\lambda_2\!+\!\psi_d(t)\!\big)\!\in\!\cF_{\varphi_d}\end{array}\!\!\!\!\!\!\!\!}\!\!,\\
\cD_v^d&\!:=\!\setdef{\!\!\!\!(t,x,v)\!\in\!\R_{\ge 0}\!\times\!\R^2\!\!\!}{\small\!\!\!\!\!\!\!\begin{array}{l} \big(t,v\!-\!v_{\rm ref}(t)\big)\!\in\!\cF_{\varphi_v}\ \wedge\ \\[1mm] \big(t, \lambda_1 v\!+\!x\!-\!x_l(t)\!+\!\lambda_2\!+\!\psi_d(t)\!\big)\!\in\!\cF_{\varphi_d}\end{array}\!\!\!\!\!\!\!\!}\!\!.
\end{align*}

In the remainder of this section we give the feasibility result for the controller~\eqref{eq:fun-con-fin} and its proof. Before we do this, we like to emphasize that funnel cruise control is not a mere application of funnel control, in particular Theorem~\ref{Thm:funnel}. While the control objective is reformulated to fit into the methodology of funnel control, here we restrict the evolution of the state variables in a much more complex way. In classical funnel control, as recalled in Section~\ref{Ssec:FunConRD1}, it is required that the tracking error $e(t) = y(t) - y_{\rm ref}(t)$ satisfies $|e(t)|< \varphi(t)^{-1}$ for some~$\varphi\in\Phi$, i.e.,~$y(t)$ is contained in the interval~$\big(y_{\rm ref}(t)-\varphi(t)^{-1},y_{\rm ref}(t)+\varphi(t)^{-1}\big)$. In funnel cruise control, the graph of any solution~$(x,v)$ of~\eqref{eq:Sys} is  restricted to~$\cD$, where for a fixed $t\ge 0$ the set $\{(x,v)\mid (t,x,v)\in\cD\}\subseteq \R^2$ is non-convex, see also Fig.~\ref{Fig:SWL}. Therefore, it was necessary to develop a completely new proof technique which is beyond classical funnel control techniques.

\begin{Thm}\label{Thm:funnel-acc}
Consider a system~\eqref{eq:Sys} with initial conditions~\eqref{eq:IC}, a favourite velocity $v_{\rm ref}\in \cW^{1,\infty}(\R_{\ge 0}\to\R)$ with  $v_{\rm ref}(t)\geq 0$ for all $t\geq 0$, position of the leader vehicle $x_l\in \cW^{1,\infty}(\R_{\ge 0}\to\R)$, safety distance $x_{\rm safe}$ as in~\eqref{eq:xsafe} with $\lambda_1,\lambda_2>0$ and funnel functions $\vp_v, \vp_d \in \Phi$, such that $\varphi_d(0)\neq 0$  and
\[
    \big(0,x^0,v^0)\in \cD.
\]
Then the funnel cruise controller~\eqref{eq:fun-con-fin} applied to~\eqref{eq:Sys} yields a closed-loop system which has a solution, and every solution can be extended to a maximal solution $(x,v):[0,\omega)\to\R^2$, $\omega\in(0,\infty]$, which has the properties:
\begin{enumerate}
  \item $\omega=\infty$;
  \item all involved signals $x$, $v$ and $u$ are bounded;
	\item there exists $\eps>0$ so that for all $t> 0$
\begin{align*}
\hspace*{-5mm}e_v(t)&\leq \varphi_v(t)^{-1}-\eps,\quad e_d(t)\leq \varphi_d(t)^{-1}-\eps, \\
\hspace*{-5mm} \eps&\leq\max\{0,\varphi_v(t)^{-1}+e_v(t)\}+\max\{0,\varphi_d(t)^{-1}+e_d(t)\}.
\end{align*}

\end{enumerate}
 \end{Thm}

\begin{proof}
\underline{\textit{Ste\smash{p} 1}:} We show existence of a maximal solution. Use~$\cD$ as in~\eqref{eq:cD}, define  $\psi_v(\cdot) := \varphi_v(\cdot)^{-1}$, $\psi_d(\cdot) := \varphi_d(\cdot)^{-1}$,
\begin{align*}
f:\R_{\ge 0}\!\times\!\R\!\to\!\R,~(t,v)\mapsto -\tfrac{1}{m}\Big(F_a(t,v)+F_g(t)+F_r(v) -\delta(t)\Big),
\end{align*}
and~$F$ as in~\eqref{eq:F} on the next page.
\begin{figure*}[h!tb]
\vspace*{2mm}
\begin{align}\label{eq:F}
F:\cD \to \R^2,\ (t,x,v)\mapsto
\begin{cases}
\begin{pmatrix} v \\ -\tfrac{1}{m} \frac{v-v_{\rm ref}(t)}{1-\varphi_v(t)^2\vert v-v_{\rm ref}(t)\vert^2}+f(t,v)\end{pmatrix}, & \mbox{if } (t,x,v)\in\cD_v\\\\
\begin{pmatrix} v \\ -\tfrac{1}{m}\frac{x - x_l(t) +\lambda_1v+\lambda_2 + \psi_d(t)}{1-\varphi_d(t)^2\vert x - x_l(t) + \lambda_1v+\lambda_2 + \psi_d(t)\vert^2}+f(t,v)\end{pmatrix}, & \mbox{if } (t,x,v)\in\cD_d,\\\\
\begin{pmatrix} v \\ \tfrac{1}{m}\min\left\{-\frac{v-v_{\rm ref}(t)}{1-\varphi_v(t)^2\vert v-v_{\rm ref}(t)\vert^2},-\frac{x - x_l(t) + \lambda_1v+\lambda_2 + \psi_d(t)}{1-\varphi_d(t)^2\vert x - x_l(t) + \lambda_1v+\lambda_2 + \psi_d(t)\vert^2}\right\}+f(t,v)\end{pmatrix}, & \mbox{if } (t,x,v)\in\cD^d_v.
\end{cases}\\[4mm]
\cline{1-1}\  \nonumber
\end{align}\vspace{-12mm}
\end{figure*}
Then the closed-loop initial-value problem~\eqref{eq:fun-con-fin},~\eqref{eq:Sys} is equivalent to
\begin{align*}
\begin{pmatrix}\dot{x}(t) \\ \dot{v}(t)\end{pmatrix}=F\big(t,x(t),v(t)\big),~~~\big(x(0), v(0)\big) = \big(x^0,v^0\big).
\end{align*}
Since~$F_g, F_a$ and~$F_r$ as in~\eqref{eq:slidingfriction2} are continuous and $\delta\in\cL^\infty(\R_{\ge 0}\to\R)$ it follows that~$f$ is measurable and locally integrable in~$t$ and continuous in~$v$. Furthermore,
$\varphi_v,~\varphi_d,~v_{\rm ref}$, and $x_l$ are continuous, thus~$F$ is measurable and locally integrable in~$t$ and continuous in~$v$. Since moreover the set~$\cD$ is  relatively open in $\R_{\ge 0}\times\R^2$ and satisfies ${(0,x^0,v^0)\in\cD}$, by~\cite[\S\,10, Thm.~\RN{20}]{Walt98} there exists a weakly differentiable solution, which can be extended to a maximal solution $(x,v):[0,\omega)\to\R^2$, $\omega\in(0,\infty]$. Furthermore, the closure of the graph of~$(x,v)$ is not a compact subset of~$\cD$.\\
For later use we divide the set $[0,\omega)$ into the two parts
\begin{equation*}
M_v\!:=\!\setdef{\!\!\!t\!\in\![0,\omega)\!\!\!}{\small\!\!\!\!\!\!\begin{array}{l} \big( t,x(t),v(t)\big)\in\cD_v \ \vee\ \Big(\big(t,x(t),v(t)\big)\in\cD_v^d \ \wedge \  \\ \min\{-k_v(t)e_v(t),-k_d(t)e_d(t)\}\!=\!-k_v(t)e_v(t)\Big)\end{array}\!\!\!\!\!\!\!\!\!}
\end{equation*}
and
\begin{equation*}
M_d\!:=\!\setdef{\!\!\!t\!\in\![0,\omega)\!\!\!}{\small\!\!\!\!\!\!\begin{array}{l} \big(t,x(t),v(t)\big)\in\cD_d \ \vee\ \Big(\big(t,x(t),v(t)\big)\in\cD_v^d \ \wedge \\ \min\{-k_v(t)e_v(t),-k_d(t)e_d(t)\}\!=\!-k_d(t)e_d(t)\Big)\end{array}\!\!\!\!\!\!\!\!\!}\!\!.
\end{equation*}

\underline{\textit{Ste\smash{p} 2}:} We show that for all $\tau\in(0,\omega)$ we have
\begin{align}
  e_v(\tau) = -\psi_v(\tau)\ \ \Longrightarrow\ \  &\exists\, \delta>0\ \forall\, t\in (\tau-\delta,\tau+\delta):\nonumber \\
   &\qquad\quad u(t) = k_d(t) e_d(t),\label{eq:ev=phiv->u=kded}\\
  e_d(\tau) = -\psi_d(\tau)\ \ \Longrightarrow\ \  &\exists\, \delta>0\ \forall\, t\in (\tau-\delta,\tau+\delta):\nonumber\\
   &\qquad\quad u(t) = k_v(t) e_v(t).\label{eq:ed=phid->u=kvev}
\end{align}
We prove~\eqref{eq:ev=phiv->u=kded}; the proof of~\eqref{eq:ed=phid->u=kvev} is analogous and omitted. First observe that $\tau\in M_d$. Since $e_d,~k_d$ are continuous in $\tau$, we have that
\[
\exists\,\delta_1>0~ \forall\, s\in(\tau-\delta_1,\tau+\delta_1):\
\vert k_d(\tau)e_d(\tau)-k_d(s)e_d(s)\vert < 1,
\]
and hence, by reverse triangle inequality,
\begin{align*}
-k_d(s)e_d(s)\leq\vert k_d(s)e_d(s)\vert<1+\vert k_d(\tau)e_d(\tau)\vert.
\end{align*}
Since $k_v(s)\to \infty$ for $s\to \tau$ and, for $\delta_1$ small enough, $e_v(s) < 0$ for $|\tau-s|< \delta_1$, we find that
\begin{multline*}
\exists\, \delta_2>0\ \forall\, s \text{ with } \big(s,x(s),v(s)\big)\in\cD_v^d,~\vert \tau-s\vert <\delta_2:\\
-k_v(s)e_v(s)> 1+\vert k_d(\tau)e_d(\tau)\vert.
\end{multline*}
Then~\eqref{eq:ev=phiv->u=kded} follows for $\delta:=\min\{\delta_1,\delta_2\}$.

\underline{\textit{Ste\smash{p} 3}:} We show that~$M_v$ and~$M_d$ are closed in~$[0,\omega)$. We prove the statement for~$M_v$; for~$M_d$ it is analogous and omitted. Let $\tau \in[0,\omega)$ be such that
\[
    \forall\, \delta>0:\ (\tau-\delta,\tau+\delta) \cap M_v \neq \emptyset.
\]
We show that $\tau\in M_v$. If $(\tau,x(\tau),v(\tau))\in\mathring\cD_v$, the interior of $\cD_v$, then clearly $\tau\in M_v$. If $e_d(\tau) = -\psi_d(\tau)$, then $\tau\in M_v$ follows from Step~2. The case $(\tau,x(\tau),v(\tau))\in\cD_d$ is not possible, because otherwise there exists $\delta>0$ such that $(\tau-\delta,\tau+\delta) \cap M_v = \emptyset$, a contradiction; this is clear for $(\tau,x(\tau),v(\tau))\in\mathring\cD_d$ and for $e_v(\tau) = -\psi_v(\tau)$ it follows from Step~2. It remains to consider the case $(\tau,x(\tau),v(\tau))\in\cD_v^d$. By assumption there exists a sequence $(t_n)_{n\in\N} \subseteq M_v$ such that $\lim_{n\to\infty} t_n = \tau$. Choose $\delta>0$ and, if necessary, a subsequence of $(t_n)$ which we again denote by $(t_n)$, such that
\begin{align*}
      & \forall\, t\in (\tau-\delta,\tau+\delta):\ (t,x(t),v(t))\in\cD_v^d\\
       \text{and}\quad &\forall\, n\in\N:\ t_n\in (\tau-\delta,\tau+\delta).
\end{align*}
If $\tau = 0$, then consider $[0,\delta)$ instead of $(\tau-\delta,\tau+\delta)$. Then we find, by continuity of~$u$,~$k_v$ and~$e_v$ on $(\tau-\delta,\tau+\delta)$ that
\[
  u(\tau) = \lim_{n\to\infty} u(t_n) \stackrel{t_n\in M_v}{=} \lim_{n\to\infty} -k_v(t_n) e_v(t_n) = -k_v(\tau) e_v(\tau),
\]
by which $\tau\in M_v$.

\underline{\textit{Ste\smash{p} 4}:} We show that~$x$ and~$v$ are bounded. The proof of the claim is divided into six steps:
\begin{itemize}
\item[1)] $v$ is bounded on $M_v$. Seeking a contradiction, assume that~$v$ is unbounded on $M_v$. Then, since~$v_{\rm ref}$ is bounded and for all $\eps>0$ we have $\left.\psi_v\right|_{[\eps,\infty)}\!\!\in\!\cL^{\infty}([\eps,\infty)\!\to\!\R)$, there exists $t^*\in M_v,~t^*>0,$  such that $\vert v(t^*) \vert> v_{\rm ref}(t^*)+\psi_v(t^*)$. This contradicts $\big(t^*,x(t^*),v(t^*)\big)\in\cD$.
\item[2)] $v$ is bounded from above on $M_d$. This can be proved similar to 1): The existence of $t^*\in M_d$, $t^*>0$, such that $v(t^*)>v_{\rm ref}(t^*)+\psi_v(t^*)$ yields a contradiction.
\item[3)] $x$ is bounded from above on $M_v$. Assuming the opposite and invoking~1) as well as boundedness of~$x_l$ gives existence of $t^*\in M_v$, $t^*>0$, so that $x(t^*)>x_l(t^*)-\lambda_1v(t^*)-\lambda_2$, which implies $\big(t^*,x(t^*),v(t^*)\big)\notin \cD$, a contradiction.
\item[4)] $x$ is bounded from above on $M_d$. Since $\big(t,x(t),v(t)\big)\in\cD$ for all $t\in[0,\omega)$, we have $\varphi_d(t)\vert x(t)-x_l(t)+\lambda_1v(t)+\lambda_2+\psi_d(t)\vert<1$ for all $t\in M_d$. By assumption~$x_l$ and $1\!/\!\varphi_d$ are bounded on~$\R_{\ge 0}$, thus $y:=\lambda_1v +x$ is bounded on $M_d$. Recalling that $\dot{x}(t)=-\tfrac{1}{\lambda_1}x(t)+\tfrac{1}{\lambda_1}y(t)$ we can estimate, with $\mu:=\tfrac{1}{\lambda_1}$, that for any $t\in [0,\omega)$
\begin{align*}
&x(t)=e^{-\mu t}x^0+\int\limits_{0}^t \mu e^{-\mu (t-s)}y(s)\ds{s}\\
&= e^{- \mu t}x^0+\!\!\!\!\int\limits_{M_v\cap [0,t]}\!\!\!\!\mu e^{-\mu (t-s)}y(s)\ds{s}+\!\!\!\!\int\limits_{M_d\cap [0,t]}\!\!\!\!\mu e^{-\mu (t-s)}y(s)\ds{s}\\
&\leq \vert x^0\vert+\!\!\!\!\int\limits_{M_v\cap [0,t]}\!\!\!\!\mu e^{-\mu (t-s)}y(s)\ds{s}+\!\!\!\!\int\limits_{M_d\cap [0,t]}\!\!\!\! \vert \mu e^{-\mu (t-s)}y(s)\vert \ds{s}.
\end{align*}
Since $y = \lambda_1v+x$ is bounded on $M_d$, there exists $M\geq0$ such that
\begin{align*}
\int\limits_{M_d\cap[0,t]}\vert \mu e^{-\mu (t-s)}y(s)\vert \ds{s}=: K_1 <\infty.
\end{align*}
Then we obtain
\begin{align*}
x(t)&\leq \vert x^0\vert+K_1+\int\limits_{M_v\cap [0,t]}\mu e^{-\mu(t-s)}(\lambda_1v(s)+x(s))\ds{s}\\
 &\overset{1),\,3)}{\le} K_2 < \infty
 \end{align*}
 for some $K_2>0$.
\item[5)] $v$ is bounded from below on $M_d$. Seeking a contradiction assume that~$v$ is unbounded from below on~$M_d$. Since~$v$ is bounded on~$M_v$ by~1), there exists a sequence $(t_k)_{k\in \N}\subseteq M_d$ with $\lim_{k\to\infty} v(t_k) = -\infty$. Since $\vert x(t_k)-x_l(t_k)+\lambda_1v(t_k)+\lambda_2+\psi_d(t_k)\vert <\psi_d(t_k)$ for all $k\in\N$, it follows that $\lim_{k\to\infty} x(t_k) = \infty$, which contradicts~4).
\item[6)] $x$ is bounded from below. Let $t\in(0,\omega)$ and first assume that $\big(t,x(t),v(t)\big)\in\overline{\cD^d_v\cup \cD_d}$. Since $x_l$ and~$1\!/\!\varphi_d$ are bounded on~$\R_{\ge 0}$, the claim immediately follows from~2) invoking the inequality
\[
x(t)\geq x_l(t)-\lambda_1v(t)-\lambda_2-2\psi_d(t)\ge K_3 >-\infty.
\]
Now let $\big(t,x(t),v(t)\big)\in\mathring{\cD_v}$. Then there exists $\eps>0$ such that $(t-\eps,t+\eps)\subseteq M_v$.
Since~$M_d$ is a closed subset of~$\R_{\ge 0}$ by Step~3, the set $\setdef{s\in M_d}{s\le t-\eps}$ is closed as well and its supremum is again an element of~$M_d$, or it equals $-\infty$. If the latter is the case, then choose $h=t$, otherwise choose $h>0$ such that $(t-h,t)\subseteq M_v$ and $t-h\in M_d$. By the mean value theorem, there exists $s\in (t-h,t)$ so that
\begin{align}
v(s)=\frac{x(t)-x(t-h)}{h}.\label{eq:v}
\end{align}
By~2) and $t-h\in M_d$ (or $t-h = 0$) we find that
\begin{align}
x(t-h)&\geq x_l(t-h)-\lambda_1v(t-h)-\lambda_2-2\psi_d(t-h)\notag\\
&\ge K_4 >-\infty.\label{eq:x(t-h)}
\end{align}
Finally, applying the assumption that $v_{\rm ref}(s)\geq 0$ for all $s\ge 0$,~\eqref{eq:v} together with~\eqref{eq:x(t-h)} implies
\begin{align*}
x(t)&=hv(s)+x(t-h)\geq h\big( \underbrace{v(s)-v_{\rm ref}(s)}_{> -\varphi_v(s)^{-1}} \big)+\underbrace{x(t-h)}_{\ge K_4}\\
&\ge K_5 >-\infty.
\end{align*}
\end{itemize}

\underline{\textit{Ste\smash{p} 5}:} We show (iii). Boundedness of $\varphi_v$ and $\varphi_d$ allows us to define
\begin{align*}
h_v&:=\inf_{t>0}\psi_v(t)>0\quad\text{and}\quad h_d:=\inf_{t>0}\psi_d(t)>0.
\end{align*}
Let $\tau\in (0,\omega)$ be arbitrary but fixed.
Since $\varphi_v\in \Phi$, we find that  $\left.\dot \psi_v\right|_{[\tau,\infty)}$ is bounded.
Thus, there exists a Lipschitz bound $L_v> 0$ of $\left.\psi_v\right|_{[\tau,\infty)}$. Similarly, we find a Lipschitz bound $L_d> 0$ of $\left.\psi_d\right|_{[\tau,\infty)}$. Since~$v$ is bounded by Step~4, it follows that $F_g(\cdot)$, $F_r\big(v(\cdot)\big)$, $F_a\big(\cdot,v(\cdot)\big)$, $\delta(\cdot)$ and $\dot{v}_{\rm ref}(\cdot)$ are bounded and hence there exists $K>0$ independent of~$\omega$ such that
 \begin{equation}\label{eq:K}
 \begin{aligned}
&\text{for almost all\ $t\in[0,\omega)$}:\\  &\left|\tfrac{1}{m}\big(F_r(v(t))+F_a(t,v(t))+F_g(t) - \delta(t)\big)+\dot{v}_{\rm ref}(t) \right| \le K.
\end{aligned}
\end{equation}
Now choose $\eps>0$ small enough so that
\begin{align*}
\eps&\leq \min\left\{\tfrac{h_v}{2},\tfrac{h_d}{2},\inf\limits_{t\in (0,\tau]}\big(\psi_v(t)\!\!-\!e_v(t)\big),\inf\limits_{t\in (0,\tau]}\big(\psi_d(t)\!\!-\!e_d(t)\big),\right.\\
&\left.\ \ \ \ \inf\limits_{t\in (0,\tau]}\Big[\max \big\{0,\psi_v(t)\!\!+\!e_v(t)\big\}+\max\big\{0,\psi_d(t)\!\!+\!e_d(t)\big\}\Big]\!\right\}
\end{align*}
and
\begin{align}
L:=\max\{L_v,L_d\}\leq -K+\frac{\min\{h_d^2,h_v^2\}}{4\eps m}.\label{eq:L}
\end{align}
We show that
\begin{equation}\label{show}
\begin{aligned}
\forall\,t\in (0,\omega):\ \ e_v(t)&\leq \psi_v(t)-\eps,\\
e_d(t)&\leq \psi_d(t)-\eps, \\
\eps&\leq\max\{0,\psi_v(t)+e_v(t)\}\\
&\quad+\max\{0,\psi_d(t)+e_d(t)\};
\end{aligned}
\end{equation}
for the first two inequalities we use a standard procedure in funnel control, see e.g.~\cite{BergHoan18}, but the proof of the third is much more involved. First note that by definition of~$\eps$ the inequalities~\eqref{show} hold on $(0,\tau]$. Seeking a contradiction, suppose there exists $t_1\in (\tau,\omega)$ such that at least one of the following three cases occurs.

\textit{Case \RN{1}:}\ $\psi_v(t_1)-e_v(t_1)<\eps$. Set
\begin{align*}
t_0:=\max\setdef{t\in [\tau,t_1)}{\psi_v(t)-e_v(t) =\eps}.
\end{align*}
Then, for all $t\in[t_0,t_1]$, we have that
\begin{equation}\label{eq:kvev1}
\begin{aligned}
\psi_v(t)-e_v(t)&\leq\eps,\\
 e_v(t) \geq \psi_v(t)-\eps\geq h_v-\eps&\geq \tfrac{h_v}{2},\\
k_v(t)=\tfrac{1}{1-\varphi_v(t)^2\vert e_v(t)\vert ^2}\geq\tfrac{1}{2\eps\varphi_v(t)}&\geq \tfrac{h_v}{2\eps}.
\end{aligned}
\end{equation}
In particular,~\eqref{eq:kvev1} together with~\eqref{eq:fun-con-fin} implies that
\begin{align}
\forall\, t\in[t_0,t_1]:\ u(t)\leq -k_v(t)e_v(t).\label{eq:u}
\end{align}
Now we calculate for almost all $t\in [t_0,t_1]$ that
\begin{align*}
&\dot{e_v}(t) =\tfrac{1}{m}\Big( u(t)-F_r\big(v(t)\big)-F_a\big(t,v(t)\big)-F_g(t) +\delta(t)\Big)-\dot{v}_{\rm ref}(t)\\
																	&\overset{\eqref{eq:u}}{\leq} \tfrac{1}{m}\Big(\!-\!k_v(t)e_v(t)\!-\!F_r\big(v(t)\big)\!-\!F_a\big(t,v(t)\big)\!-\!F_g(t) +\delta(t)\Big)\!-\!\dot{v}_{\rm ref}(t)\\
																	 &\overset{\eqref{eq:K}}{\leq} K-\tfrac{1}{m}k_v(t) e_v(t)																\overset{\eqref{eq:kvev1}}{\leq} K-\tfrac{1}{m}\tfrac{h_v}{2\eps}\tfrac{h_v}{2}\overset{\eqref{eq:L}}{\leq} -L\leq -L_v.
\end{align*}
Therefore,
\begin{align*}
e_v(t_1)-e_v(t_0) &=\int_{t_0}^{t_1} \dot e_v(t)\ds{t}
                                       \leq -L_v(t_1-t_0) \\&\leq -\vert \psi_v(t_1)-\psi_v(t_0)\vert \leq \psi_v(t_1)-                                        \psi_v(t_0)\!.
\end{align*}
Hence, $\eps=\psi_v(t_0)- e_v(t_0) \leq\psi_v(t_1)- e_v(t_1) <\eps$, a contradiction.

\textit{Case \RN{2}:}\ $\psi_d(t_1)-e_d(t_1)<\eps$. The proof is analogous to \textit{Case \RN{1}} and omitted.

\textit{Case \RN{3}:}\ $\max \{0,\psi_v(t_1)+e_v(t_1)\}+\max\{0,\psi_d(t_1)+e_d(t_1)\}<\eps$. Set
\begin{align*}
t_0:=\max\setdef{t\in [\tau,t_1)}{\!\!\!\begin{array}{l} \max \{0,\psi_v(t)+e_v(t)\}\\+\max\{0,\psi_d(t)+e_d(t)\}=\eps\end{array}\!\!\!\!\!}.
\end{align*}
Now we distinguish another three cases.\\
\textit{Case \RN{3}.A:}\ $0<\psi_d(t_0)+e_d(t_0)=\eps$ and $\psi_v(t_0)+e_v(t_0)\leq 0$. By continuity and definition of $t_0$, there exists $\tilde{t}_1\in (t_0,t_1)$ such that
\begin{align*}
\forall\, t\in(t_0,\tilde{t}_1]:\ 0<\psi_d(t)+e_d(t) <\eps.
\end{align*}
Since $e_v(t_0) \le -\psi_v(t_0)$ it further follows from Step~2 and~\eqref{eq:fun-con-fin} that there exists $\hat t_1\in (t_0,\tilde t_1]$ such that
\begin{align}
\forall\, t\in[t_0,\hat{t}_1]:\ u(t) = -k_d(t) e_d(t).\label{eq:u3}
\end{align}
Moreover, we have for all $t\in[t_0,\hat{t}_1]$ that
\begin{equation}\label{eq:kvev3}
\begin{aligned}
\psi_d(t)+e_d(t)&\leq\eps,\\
-e_d(t)\geq \psi_d(t)-\eps\geq h_d-\eps&\geq \tfrac{h_d}{2},\\
k_d(t)=\tfrac{1}{1-\varphi_d(t)^2\vert e_d(t)\vert ^2}\geq\tfrac{1}{2\eps\varphi_d(t)}&\geq \tfrac{h_d}{2\eps}.
\end{aligned}
\end{equation}
Now we calculate for almost all $ t\in [t_0,\hat{t}_1]$ that
\begin{align*}
&-\dot{e_d}(t)=\tfrac{1}{m}\Big(-u(t)+F_r\big(v(t)\big)+F_a\big(t,v(t)\big)+F_g(t) -\delta(t)\Big)\\
&\qquad\qquad+\dot{v}_{\rm ref}(t)\\
																	&\overset{\eqref{eq:u3}}{=} \tfrac{1}{m}\Big(k_d(t)e_d(t)\!+\!F_r\big(v(t)\big)\!+\!F_a\big(t,v(t)\big)\!+\!F_g(t)-\delta(t)\Big)\!+\!\dot{v}_{\rm ref}(t)\\
																	 &\overset{\eqref{eq:K}}{\leq} K + \tfrac{1}{m}k_d(t)e_d(t)
																	\overset{\eqref{eq:kvev3}}{\leq} K-\tfrac{1}{m}\tfrac{h_d}{2\eps}\tfrac{h_d}{2}\overset{\eqref{eq:L}}{\leq} -L\leq -L_d.
\end{align*}
Therefore,
\begin{align*}
e_d(t_0)-e_d(\hat{t}_1)&=\int_{t_0}^{\hat{t}_1} -\dot e_d(t)\,\ds{t}
                                       \leq -L_d(\hat{t}_1-t_0) \\&\leq -\vert \psi_d(\hat{t}_1)\!-\psi_d(t_0)\vert \leq \psi_d(\hat{t}_1)\!-                                        \psi_d(t_0)\!.
\end{align*}
Hence, $\eps=\psi_d(t_0)+e_d(t_0)\leq\psi_d(\hat{t}_1)+e_d(\hat{t}_1)<\eps$, a contradiction.\\
\textit{Case \RN{3}.B:}\ $0<\psi_v(t_0)+e_v(t_0)=\eps$ and $\psi_d(t_0)+e_d(t_0)\leq 0$. The proof is analogous to \textit{Case \RN{3}.A} and omitted.\\
\textit{Case \RN{3}.C:}\ $0<\psi_v(t_0)+e_v(t_0)<\eps$ and $0<\psi_d(t_0)+e_d(t_0)<\eps$. Then, by definition of $t_0$ and continuity, there exists $\hat{t}_1\in (t_0,t_1)$ such that either
\begin{gather}
\forall\, t\in(t_0,\hat{t}_1]:\ 0\!<\!\psi_d(t)\!+\!e_d(t)\!<\!\psi_d(t_0)\!+\!e_d(t_0)\!<\!\eps\label{eq:C.e_v}\\
{\rm or}\nonumber\\
\forall\, t\in(t_0,\hat{t}_1]:\ 0\!<\!\psi_v(t)\!+\!e_v(t)\!<\!\psi_v(t_0)\!+\!e_v(t_0)\!<\!\eps.\label{eq:C.e_d}
\end{gather}
Furthermore, if~\eqref{eq:C.e_v}, then also $0<\psi_v(t)+e_v(t)<\eps$ holds for all $t\in(t_0,\hat t_1)$, and if~\eqref{eq:C.e_d}, then also $0<\psi_d(t)+e_d(t)<\eps$ holds for all $t\in(t_0,\hat t_1)$. Therefore, in both cases~\eqref{eq:C.e_v} and~\eqref{eq:C.e_d} we have that~\eqref{eq:kvev3} and
\begin{equation}\label{eq:kvev4}
\begin{aligned}
\psi_v(t)+e_v(t)&\leq\eps,\\
-e_v(t)\geq \psi_v(t)-\eps\geq h_v-\eps&\geq \tfrac{h_v}{2},\\
k_v(t)=\tfrac{1}{1-\varphi_v(t)^2\vert e_v(t)\vert ^2}\geq\tfrac{1}{2\eps\varphi_v(t)}&\geq \tfrac{h_v}{2\eps}
\end{aligned}
\end{equation}
hold. Then, in both cases, we calculate for all $t\in [t_0,\hat{t}_1]$ that
\[
    -k_d(t) e_d(t) \stackrel{\eqref{eq:kvev3}}{\ge} \tfrac{h_d^2}{4\eps}\quad \wedge\quad -k_v(t) e_v(t) \stackrel{\eqref{eq:kvev4}}{\ge} \tfrac{h_v^2}{4\eps},
\]
and by~\eqref{eq:fun-con-fin} this implies that
\begin{equation}\label{eq:est-u-lower}
    u(t) \ge \frac{\min\{h_d^2, h_v^2\}}{4\eps}.
\end{equation}
Then, similar as in \textit{Case \RN{3}.A}, we find that
\begin{align*}
&-\dot{e_d}(t)=\tfrac{1}{m}\Big(-u(t)+F_r\big(v(t)\big)+F_a\big(t,v(t)\big)+F_g(t)-\delta(t)\Big)\\
&+\dot{v}_{\rm ref}(t) \overset{\eqref{eq:K}}{\leq} K - \tfrac{u(t)}{m} \stackrel{\eqref{eq:est-u-lower}}{\le}  K-\frac{\min\{h_d^2, h_v^2\}}{4m \eps}\overset{\eqref{eq:L}}{\leq} -L\leq -L_d.
\end{align*}
Hence, again we obtain
\[
    e_d(t_0)-e_d(\hat{t}_1) \leq - L_d (\hat{t}_1-t_0) \le \psi_d(\hat{t}_1)- \psi_d(t_0),
\]
by which $\psi_d(t_0)+e_d(t_0)\leq\psi_d(\hat{t}_1)+e_d(\hat{t}_1)$, which contradicts~\eqref{eq:C.e_v}. Similarly, we may calculate $-\dot e_v(t) \le -L_v$, from which we infer $\psi_v(t_0)+e_v(t_0)\leq\psi_v(\hat{t}_1)+e_v(\hat{t}_1)$, which contradicts~\eqref{eq:C.e_d}.

\underline{\textit{Ste\smash{p} 6}:} We show that~$u$ is bounded. First observe that for $t_0, t\in(0,\omega)$ such that $t\ge t_0$ and $(t,x(t),v(t))\in\cD_v$ we have by~\eqref{eq:fun-con-fin} that $u(t) = -k_v(t) e_v(t)$ and by Step~5 that $|e_v(t)| \le \psi_v(t) - \eps$. Therefore, we obtain
\begin{align*}
  |u(t)| &= \frac{|e_v(t)|}{(1-\varphi_v(t) |e_v(t)|) (1+\varphi_v(t) |e_v(t)|)} \le \frac{|e_v(t)|}{1-\varphi_v(t) |e_v(t)|}\\
  &\le  \frac{|e_v(t)|}{\varphi_v(t) \eps} \le \frac{1}{\eps}\, \sup_{s\ge t_0} \psi_v(s)^{2}.
\end{align*}
Analogously, for $t\in(0,\omega)$ such that $(t,x(t),v(t))\in\cD_d$ we find that
\[
    |u(t)| \le \frac{1}{\eps}\, \sup_{s\ge 0} \psi_d(s)^{2}.
\]
Now let $t_0, t\in(0,\omega)$ such that $t\ge t_0$ and $(t,x(t),v(t))\in\cD_v^d$. We show that
\[
    |u(t)| \le C_{t_0} := \frac{2}{\eps}\, \max\left\{ \sup_{s\ge t_0} \psi_v(s)^{2}, \sup_{s\ge 0} \psi_d(s)^{2}\right\}.
\]
If $e_v(t)>0$ and $e_d(t)<0$, then $u(t) = -k_v(t) e_v(t)$ by~\eqref{eq:fun-con-fin} and hence we find, similar as above, that $|u(t)| \le C_{t_0}$ is true. Analogously, for $e_v(t)<0$ and $e_d(t)>0$ we have $u(t) = -k_d(t) e_d(t)$ and thus $|u(t)| \le C_{t_0}$ is true as well. In the case $e_v(t)>0$ and $e_d(t)>0$ we have $\psi_v(t) - e_v(t) \ge \eps$ and $\psi_d(t) - e_d(t) \ge \eps$, which gives $k_v(t) \le 1/\big(\eps \varphi_v(t)\big)$ and $k_d(t) \le 1/\big(\eps \varphi_d(t)\big)$, from which we obtain
\[
    -k_v(t) e_v(t) \ge - \frac{1}{\eps \varphi_v(t)^2}\quad\text{and}\quad -k_d(t) e_d(t) \ge - \frac{1}{\eps \varphi_d(t)^2}.
\]
As a consequence
\begin{align*}
   0&\ge u(t) = \min\{-k_v(t) e_v(t), -k_d(t) e_d(t)\}\\
   &\ge \min\left\{ - \frac{1}{\eps \varphi_v(t)^2}, - \frac{1}{\eps \varphi_d(t)^2}\right\} \ge -C_{t_0},
\end{align*}
by which we have $|u(t)| \le C_{t_0}$. Finally, consider the case $e_v(t)<0$ and $e_d(t)<0$. Then we have by Step~5 that $\eps\leq\max\{0,\psi_v(t)+e_v(t)\} +\max\{0,\psi_d(t)+e_d(t)\}$. Using $(t,x(t),v(t))\in\cD_v^d$ this implies
\[
    \max\{ \psi_v(t)+e_v(t), \psi_d(t)+e_d(t)\} \ge \frac{\eps}{2}.
\]
Therefore, we have
\[
    \psi_v(t)+e_v(t) \ge \frac{\eps}{2}\quad\text{or}\quad \psi_d(t)+e_d(t) \ge \frac{\eps}{2}.
\]
If the first statement is true, then
\begin{align*}
     k_v(t) &= \frac{1}{(1-\varphi_v(t) e_v(t)) (1+\varphi_v(t) e_v(t))} \\
     &\le \frac{1}{1+\varphi_v(t) e_v(t)} \le \frac{2}{\varphi_v(t) \eps},
\end{align*}
and if the latter is true, then $k_d(t) \le 2/\big(\varphi_d(t) \eps\big)$. Therefore,
\[
    -k_v(t) e_v(t) \le \frac{2}{\eps \varphi_v(t)^2}\quad\text{or}\quad -k_d(t) e_d(t) \le \frac{2}{\eps \varphi_d(t)^2},
\]
which implies that
\begin{align*}
  0 &\le u(t) = \min\{-k_v(t) e_v(t), -k_d(t) e_d(t)\}\\
  & \le \max\left\{ \frac{2}{\eps \varphi_v(t)^2},  \frac{2}{\eps \varphi_d(t)^2}\right\} \le C_{t_0}.
\end{align*}

\underline{\textit{Ste\smash{p} 7}:} We show that $\omega=\infty$. Note by Step~4 we have that $(x,v):[0,\omega)\to\R^2$ is bounded. Together with~(iii), which we have shown in Step~5, the assumption $\omega<\infty$ implies that the closure of the graph of~$(x,v)$ is a compact subset of~$\cD$, which contradicts the statement found in Step~1. Hence, $\omega=\infty$ and this completes the proof of the theorem.
\end{proof}

While the funnel cruise controller~\eqref{eq:fun-con-fin} is robust and able to guarantee safety at all times, we like to emphasize that, although~$u(\cdot)$ is bounded, it is not possible to respect any control constraints of the form $u_{\min} \le u(t) \le u_{\max}$, which are typically present in any real-world application. However, for practical implementation we may simply use the saturated signal
\[
    \hat u(t) = \begin{cases} u_{\max}, & \mbox{ if $u(t)\ge u_{\max}$},\\ u(t), & \mbox{ if $u_{\min} < u(t) < u_{\max}$}\\ u_{\min}, & \mbox{ if $u(t)\le u_{\min}$},   \end{cases}
\]
where~$u(t)$ is as in~\eqref{eq:fun-con-fin}, as the real control input. Of course, this may cause the velocity or distance error to leave the respective performance funnel, but if the two vehicles have comparable parameters (in particular, comparable masses), then they will evolve back into the respective funnel after some short time. If the masses of the vehicles differ a lot this may not be true. For instance, a truck cannot decelerate as fast as a passenger car.

The safety distance may be violated while~$u(t)$ is saturated, but this does not automatically mean that a collision occurs. Future research is necessary to identify ranges for the parameters of the leader and follower vehicle so that collision avoidance can still be guaranteed.

\section{Simulations}\label{Sec:Sim}

We illustrate the funnel cruise controller~\eqref{eq:fun-con-fin} for three different scenarios which may occur in daily traffic. The first standard scenario is that the follower vehicle, with a constant favourite velocity~$v_{\rm ref}$, is far away from the leader, catches up and follows it for some time until the leader accelerates past~$v_{\rm ref}$. The second scenario illustrates that safety is guaranteed even in the case of a full brake of the leader vehicle. In the last scenario the leader vehicle has a strongly varying acceleration.

For the simulations we will use some typical parameter values for the model~\eqref{eq:Sys} which are taken from~\cite{AstrMurr08} and summarized in Table~\ref{Tab:Param}.

\begin{table}[h!tb]
\centering
\begin{tabular}{|c|c|c|c|c|c|}
  \hline
    $m$ & $\theta(t)$ & $\rho(t)$ & $C_d$ & $C_r$ & $A$  \\
   \hline
    $\SI{1300}{\kilo\gram}$ & $\SI{0}{\radian}$ 
     & $\SI{1.3}{\kilo\gram\per\metre^3}$ & 0.32 & 0.01 & $\SI{2.4}{\metre^2}$\\
  \hline
\end{tabular}
\caption{Parameter values for the vehicle model.}
\label{Tab:Param}
\end{table}


The disturbance is chosen as $\delta(\cdot) = 0$ and for the approximated friction model~\eqref{eq:slidingfriction2} we choose the parameter $\alpha = 100$. The initial conditions~\eqref{eq:IC} are chosen as $x^0=\SI{0}{\metre}$ and $v^0 = \SI{15}{\metre\per\second}$ and the constants in~\eqref{eq:xsafe} as $\lambda_1 = \SI{0.5}{\second}$ and $\lambda_2 = \SI{2}{\metre}$. For all three scenarios we choose the favourite velocity~$v_{\rm ref}(t) = \SI{36}{\metre\per\second}$ and the funnel functions
\[
  \varphi_v(t) = \big( 22.5 e^{-0.2 t} + 0.2\big)^{-1},\quad
  \varphi_d(t) = 0.25.
\]

\textbf{Scenario 1}: We have chosen~$x_l$ and~$v_l=\dot x_l$ so that initially the leader vehicle has a larger velocity than the follower, which is hence free to accelerate and catch up using the velocity funnel controller. When the distance is between $x_{\rm safe}(t) + 2\psi_d(t)$ and $x_{\rm safe}(t)$, the distance funnel controller will ensure that the safety distance is not violated. After a period of safe following, where $v(t) < v_{\rm ref}(t) - \vp_v(t)^{-1}$, the leader accelerates to a velocity larger than $v_{\rm ref}(t)$ and the velocity funnel controller will again take over. The simulation of the controller~\eqref{eq:fun-con-fin} for the system~\eqref{eq:Sys} with parameters as in Table~\ref{Tab:Param} and the above described scenario over the time interval $0-50\si{\second}$ has been performed in MATLAB (solver: {\tt ode15s}, rel.\ tol.: $10^{-10}$, abs.\ tol.: $10^{-10}$) and is depicted in Fig.~\ref{fig:sim1}.

\captionsetup[subfloat]{labelformat=empty}
\begin{figure}[h!tb]
  \centering
  \subfloat[Fig.~\ref{fig:sim1}a: Distance to leader and distance funnel]
{
\centering
  \includegraphics[width=9cm]{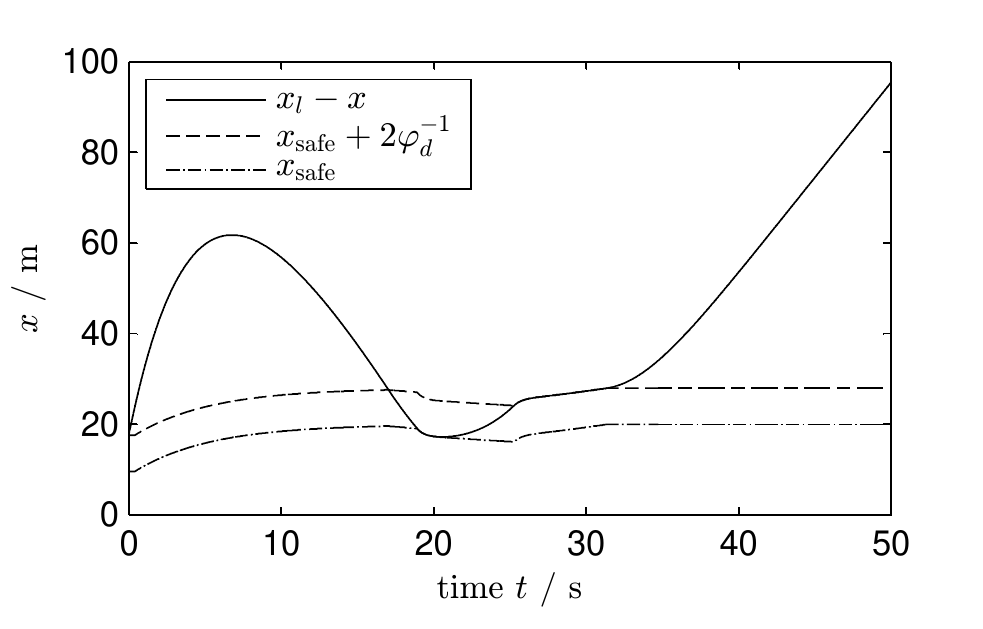}
\label{fig:sim-e}
}\\[-1mm]
\subfloat[Fig.~\ref{fig:sim1}b: Velocities and velocity funnel]
{
\centering
  \includegraphics[width=9cm]{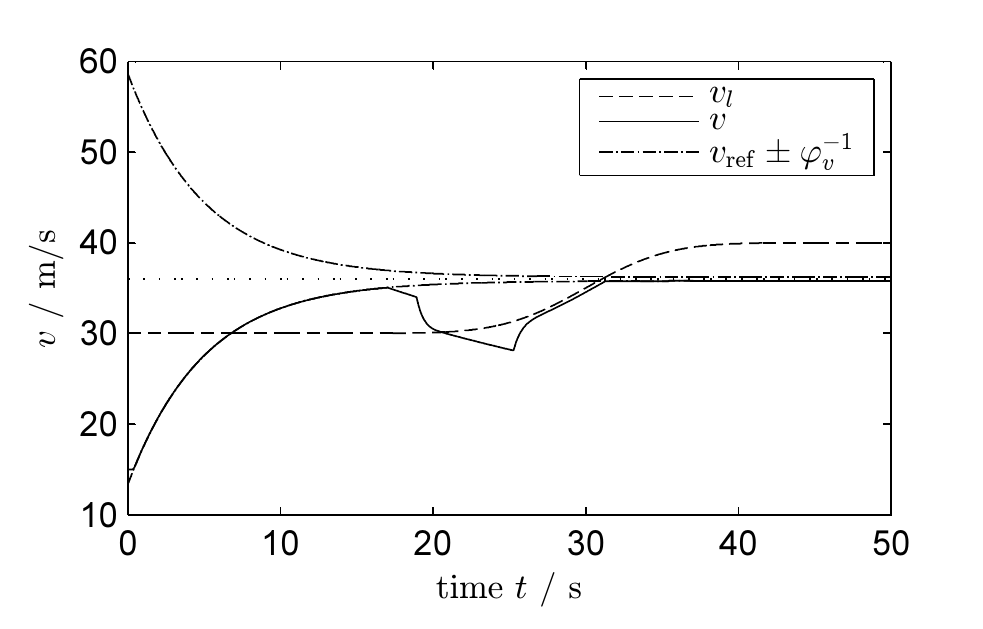}
\label{fig:sim-u}
}
\\[-1mm]
\subfloat[Fig.~\ref{fig:sim1}c: Input and leader engine force]
{
\centering
  \includegraphics[width=9cm]{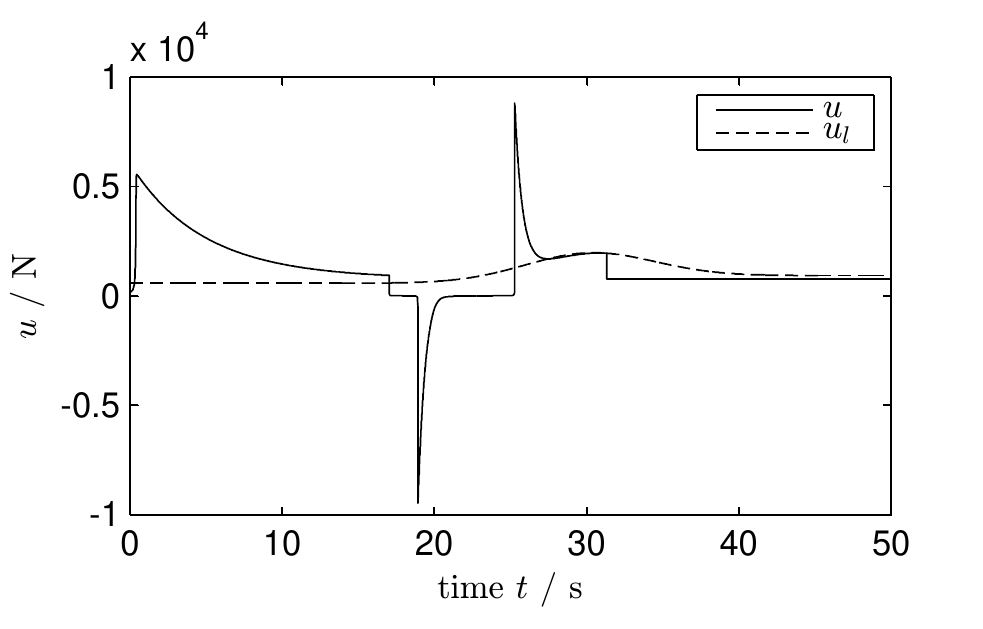}
\label{fig:sim-k}
}
\caption{Simulation of the funnel cruise controller~\eqref{eq:fun-con-fin} for the system~\eqref{eq:Sys} with parameters as in Table~\ref{Tab:Param} in Scenario~1.}
\label{fig:sim1}
\end{figure}

Fig.~\ref{fig:sim1}a shows the distance~$x_l-x$ to the leader and the distance funnel in front of the safety distance. The velocities~$v$ and~$v_l$ are depicted in Fig.~\ref{fig:sim1}b together with the velocity funnel. Fig.~\ref{fig:sim1}c shows the input signal~$u$ generated by the controller and the engine force~$u_l$ of the leader vehicle. We stress that, due to the mass of the vehicles of $\SI{1300}{\kilo\gram}$, the forces~$u$ and~$u_l$ which are between $\pm 10^4\si{\newton}$ correspond to an acceleration between $\pm \SI{8}{\metre\per\second^2}$. It can be seen that the controller achieves the favourite velocity as far as possible while guaranteeing safety at all times, thus the control objectives~(O1) and~(O2) are satisfied.

\textbf{Scenario 2}: We have chosen~$x_l$ and~$v_l$ so that after a period of safe following the leader vehicle suddenly fully brakes. This illustrates that even in such extreme cases the funnel cruise controller is able to guarantee that the safety distance is not violated.

\captionsetup[subfloat]{labelformat=empty}
\begin{figure}[h!tb]
  \centering
  \subfloat[Fig.~\ref{fig:sim2}a: Distance to leader and distance funnel]
{
\centering
  \includegraphics[width=9cm]{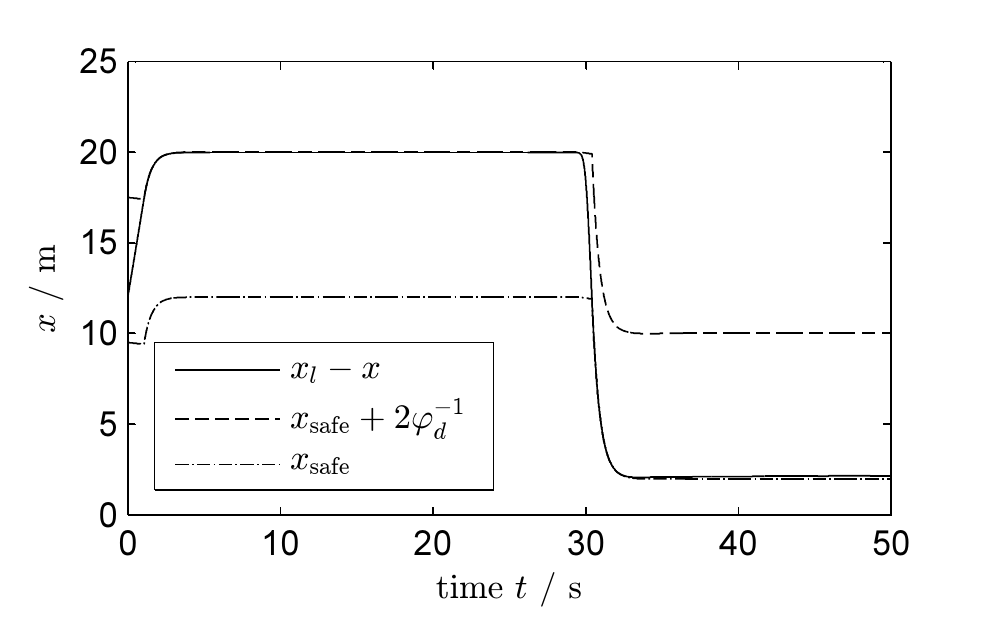}
\label{fig:sim-e}
}\\[-1mm]
\subfloat[Fig.~\ref{fig:sim2}b: Velocities and velocity funnel]
{
\centering
  \includegraphics[width=9cm]{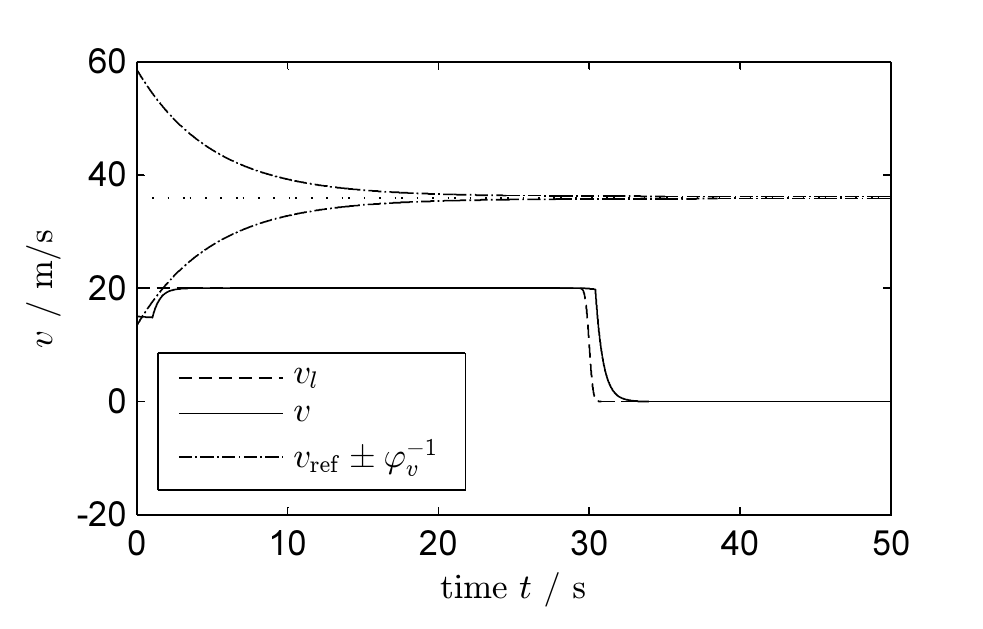}
\label{fig:sim-u}
}
\\[-1mm]
\subfloat[Fig.~\ref{fig:sim2}c: Input and leader engine force]
{
\centering
  \includegraphics[width=9cm]{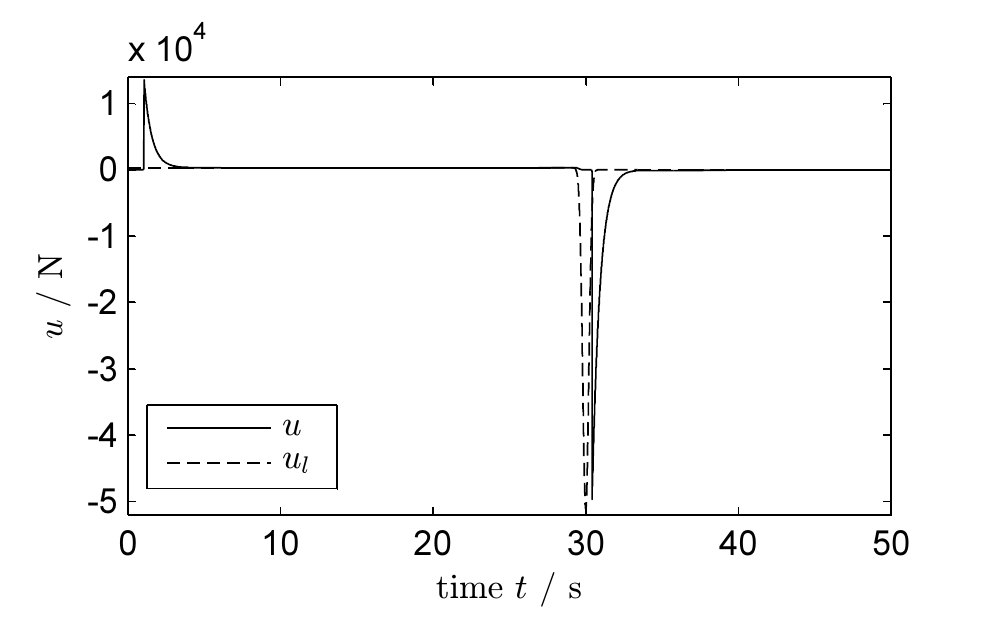}
\label{fig:sim-k}
}
\caption{Simulation of the funnel cruise controller~\eqref{eq:fun-con-fin} for the system~\eqref{eq:Sys} with parameters as in Table~\ref{Tab:Param} in Scenario~2.}
\label{fig:sim2}
\end{figure}

The simulation of the controller~\eqref{eq:fun-con-fin} for the system~\eqref{eq:Sys} with parameters as in Table~\ref{Tab:Param} and the above described scenario over the time interval $0-50\si{\second}$ has been performed in MATLAB (solver: {\tt ode15s}, rel.\ tol.: $10^{-10}$, abs.\ tol.: $10^{-10}$) and is depicted in Fig.~\ref{fig:sim2}. It can be seen in Fig.~\ref{fig:sim2}a that the safety distance is always guaranteed. From Fig.~\ref{fig:sim2}b we can observe that the velocity~$v$ of the follower does not drop as sharp as the velocity~$v_l$ of the leader. Fig.~\ref{fig:sim2}c shows that the engine forces~$u$ and~$u_l$ are quite comparable.

\textbf{Scenario 3}: We have chosen~$x_l$ and~$v_l$ so that the leader vehicle has a strongly varying acceleration. After a period of velocity control where the follower vehicle is free to accelerate close to the favourite velocity $v_{\rm ref}$, it will catch up with the leader vehicle and a period of safe following using distance funnel control follows. During this period several (sharp) acceleration and deceleration maneuvers are necessary to guarantee safety in the face of the mercurial behavior of the leader vehicle.

\captionsetup[subfloat]{labelformat=empty}
\begin{figure}[h!tb]
  \centering
  \subfloat[Fig.~\ref{fig:sim3}a: Distance to leader and distance funnel]
{
\centering
  \includegraphics[width=9cm]{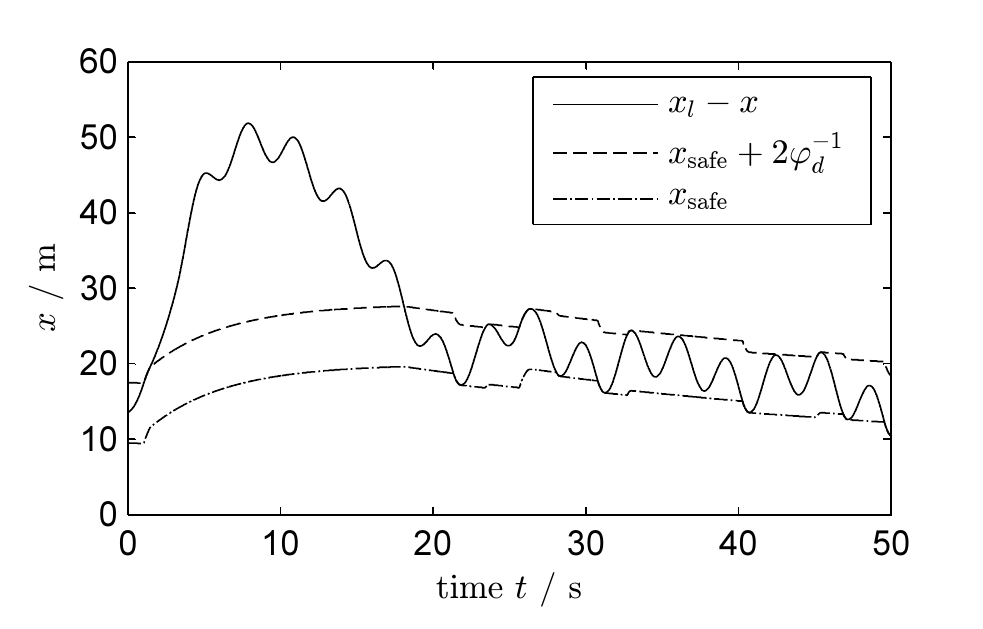}
\label{fig:sim-e}
}\\[-1mm]
\subfloat[Fig.~\ref{fig:sim3}b: Velocities and velocity funnel]
{
\centering
  \includegraphics[width=9cm]{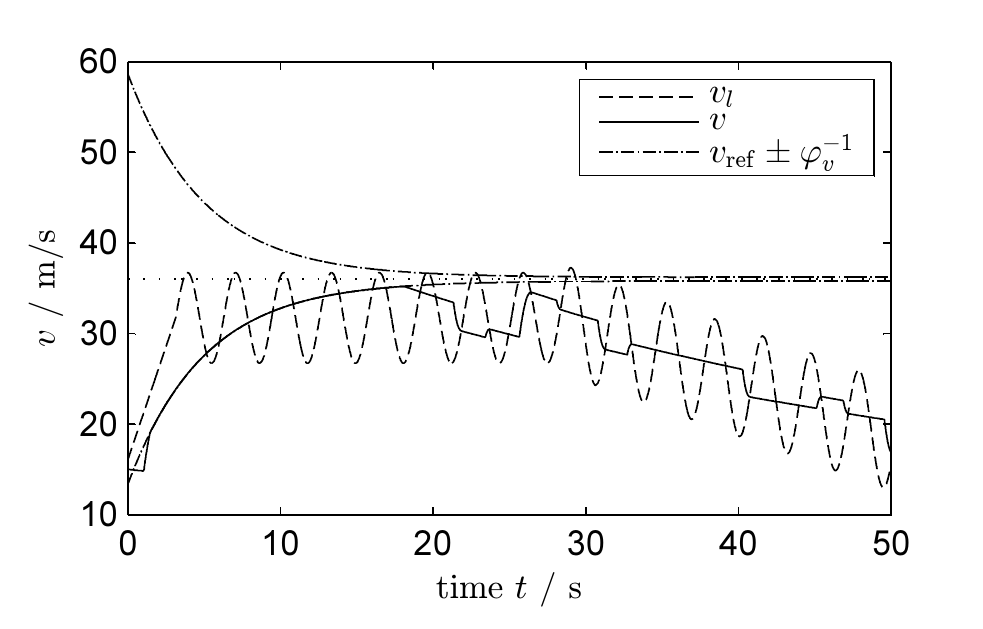}
\label{fig:sim-u}
}
\\[-1mm]
\subfloat[Fig.~\ref{fig:sim3}c: Input and leader engine force]
{
\centering
  \includegraphics[width=9cm]{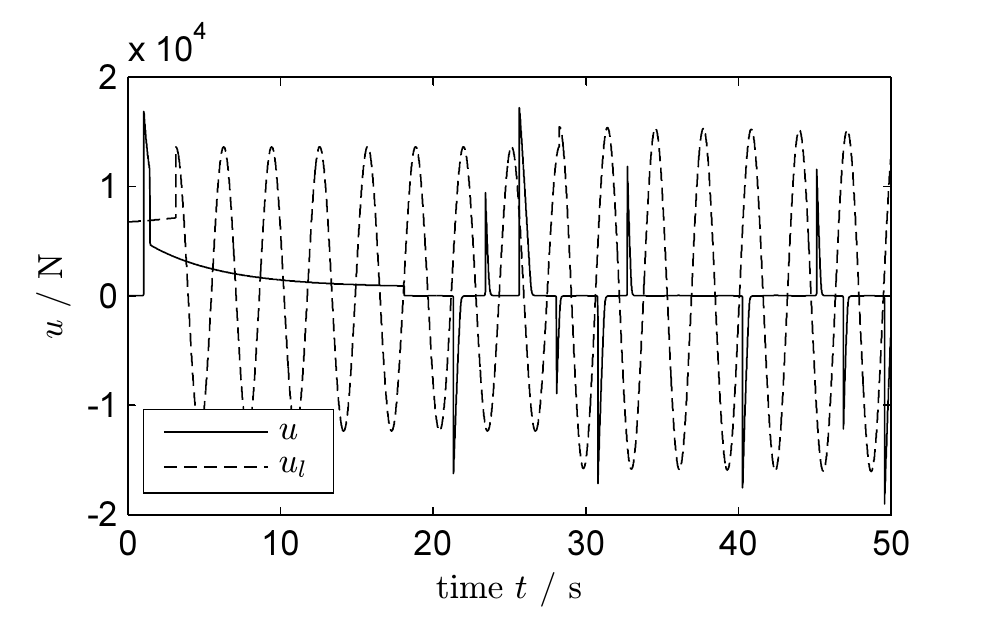}
\label{fig:sim-k}
}
\caption{Simulation of the funnel cruise controller~\eqref{eq:fun-con-fin} for the system~\eqref{eq:Sys} with parameters as in Table~\ref{Tab:Param} in Scenario~3.}
\label{fig:sim3}
\end{figure}

The simulation of the controller~\eqref{eq:fun-con-fin} for the system~\eqref{eq:Sys} with parameters as in Table~\ref{Tab:Param} and the above described scenario over the time interval $0-50\si{\second}$ has been performed in MATLAB (solver: {\tt ode15s}, rel.\ tol.: $10^{-10}$, abs.\ tol.: $10^{-10}$) and is depicted in Fig.~\ref{fig:sim3}. It can be seen in Fig.~\ref{fig:sim3}a that safety is guaranteed even in the case of the strongly varying behavior of the leader. The velocity~$v$ of the follower, as shown in Fig.~\ref{fig:sim3}b, does not vary as strongly as~$v_l$, but shows a much smoother behavior. The engine force~$u$ depicted in Fig.~\ref{fig:sim3}c shows sharp drops and rises, but this cannot be avoided since the funnel cruise controller~\eqref{eq:fun-con-fin} is causal and hence cannot look into the future. This is different from other approaches such as MPC (see e.g.~\cite{BageGarr04,LiLi11,MagdAlth17}) which is able to incorporate future information since an optimal control problem is solved over some future time interval. However, the drawback of this is that the model~\eqref{eq:Sys} must be known as good as possible.

%
\section{Conclusion}\label{Sec:Concl}
%

In the present paper we proposed a novel and universal adaptive cruise control mechanism which is robust, model-free and guarantees safety at all times. The funnel cruise controller consists of a velocity funnel controller, which is active when the leader vehicle is far away, and a distance funnel controller, which ensures that the safety distance is not violated when the leader vehicle is close. We have given a rigorous proof of feasibility of this controller. Three simulation scenarios illustrate the application of the funnel cruise controller.

The simulations show that, although the funnel cruise controller satisfies the control objectives, the generated control input (engine force of the follower vehicle) usually contains sharp peaks which are not desired in terms of driver comfort. This issue should be resolved by smoothing the peaks e.g.\ by combining the funnel cruise controller with the PI-funnel controller with anti-windup as discussed in~\cite{Hack13}. Furthermore, control constraints are typically present in real-world applications and it should be investigated for which parameters of the vehicles a saturation of the control input can still guarantee that collisions are avoided.

Another topic of future research is the investigation of platoons of several vehicles, each equipped with a funnel cruise controller. An important question is under which conditions string stability of the platoon is achieved and whether some communication between the vehicles must be allowed for this.

\vspace{-2mm}



\bibliographystyle{elsarticle-harv}

\end{document}